\documentclass[11pt, reqno]{amsart}

\usepackage{amsmath, amsfonts, amssymb, amsbsy, bigstrut, graphicx, enumerate,  upref, longtable, comment, comment, csquotes, booktabs, array, verbatim}
\usepackage{booktabs}
\usepackage{tabularx}
\usepackage{tikz}
\usepackage{float}
\usepackage{subcaption}
\usepackage{pstricks}

\usepackage[breaklinks]{hyperref}

\usepackage[numbers, sort&compress]{natbib}

\newcommand{\mg}{\mathcal{G}}

\newcommand{\mc}{\mathcal{C}}

\newtheorem{thm}{Theorem}[section]
\newtheorem{lmm}[thm]{Lemma}
\newtheorem{cor}[thm]{Corollary}
\newtheorem{prop}[thm]{Proposition}

\theoremstyle{definition}

\newtheorem{ex}[thm]{Example}

\newcommand{\bbw}{\mathbf{W}}
\newcommand{\bbx}{\mathbf{X}}

\newcommand{\bbz}{\mathbf{Z}}

\newcommand{\bx}{\mathbf{x}}
\newcommand{\by}{\mathbf{y}}

\newcommand{\ee}{\mathbb{E}}

\newcommand{\mf}{\mathcal{F}}

\newcommand{\cp}{\mathcal{P}}

\newcommand{\pp}{\mathbb{P}}

\newcommand{\rr}{\mathbb{R}}

\newcommand{\var}{\mathrm{Var}}
\newcommand{\ve}{\varepsilon}

\newcommand{\fpar}[2]{\frac{\partial #1}{\partial #2}}

\newcommand{\bz}{\mathbf{z}}

\newcommand{\fc}{\mathfrak{C}}
\newcommand{\fs}{\mathfrak{S}}

\numberwithin{equation}{section}

\newcommand{\pkg}[1]{{\normalfont\fontseries{m}\selectfont #1}}

\begin{document}
\title{A simple measure of conditional dependence}
\author{Mona Azadkia}
\author{Sourav Chatterjee}
\address{Department of Statistics, Stanford University, Sequoia Hall, 390 Jane Stanford Way, Stanford, CA 94305}
\email{mazadkia@stanford.edu}
\email{souravc@stanford.edu}
\thanks{Research partially supported by NSF grants DMS-1608249 and DMS-1855484}
\begin{abstract}
We propose a coefficient of conditional dependence between two random variables $Y$ and $Z$ given a set of other variables $X_1,\ldots,X_p$, based on an i.i.d.~sample. The coefficient has a long list of desirable properties, the most important of which is that under absolutely no distributional assumptions, it converges to a limit in $[0,1]$, where the limit is  $0$ if and only if $Y$ and $Z$ are conditionally independent given $X_1,\ldots,X_p$, and is $1$ if and only if  $Y$ is equal to a measurable function of $Z$ given $X_1,\ldots,X_p$.  Moreover, it has a natural interpretation as a nonlinear generalization of the familiar partial $R^2$ statistic for measuring conditional dependence by regression. Using this statistic, we devise a new variable selection algorithm, called Feature Ordering by Conditional Independence (FOCI), which is model-free, has no tuning parameters, and is provably consistent under sparsity assumptions.  A number of applications to synthetic and real datasets are worked out. 
\end{abstract}
\maketitle

\section{Introduction}\label{introsec}

The problem of measuring the amount of dependence between two random variables  is an old problem in statistics. Numerous methods have been proposed over the years. For recent surveys, see \cite{chatterjee19, Josse16}. The literature on measures of {\it conditional dependence}, on the other hand, is not so large, especially in the non-parametric setting.  

The non-parametric conditional independence testing problem can be relatively easily solved for discrete data using the classical Cochran--Mantel--Haenszel test \cite{cochran54, Mantel59}. This test can be adapted for continuous random variables by binning the data~\cite{huang10} or using kernels~\cite{Fukumizu08, zhang11, strobl19, doran14, sen17}. 

Besides these, there are methods based on estimating conditional cumulative distribution functions~\cite{Linton96, pss16},  conditional characteristic functions~\cite{Su07}, conditional probability density functions~\cite{Su08}, empirical likelihood~\cite{Su14}, mutual information and entropy~\cite{runge18, Joe89, Poczos12}, copulas~\cite{bergsma04, song09, veraverbeke11}, distance correlation~\cite{Wang15, Fan15, szekelyrizzo14}, and other approaches~\cite{Seth12}. A number of interesting ideas based on resampling and permutation tests have been proposed in recent years~\cite{Candes18, sen17, berrett19}. 

The first contribution of this paper is a new coefficient of conditional dependence between two random variables $Y$ and $Z$ given a set  of other variables $X_1,\ldots,X_p$, based on i.i.d.~data. The coefficient is inspired by  a similar measure of univariate dependence recently proposed in \cite{chatterjee19}. The main features of our coefficient are the following: 
\begin{enumerate}
\item it has a simple expression,
\item it is fully non-parametric, 
\item it has no tuning parameters,
\item there is no need for estimating conditional densities, conditional characteristic functions, or mutual information, 
\item it can be estimated from data very quickly, in time $O(n\log n)$ where $n$ is the sample size, 
\item asymptotically, it converges to a limit in $[0,1]$, where the limit is  $0$ if and only if $Y$ and $Z$ are conditionally independent given $X_1,\ldots,X_p$, and is $1$ if and only if  $Y$ is equal to a measurable function of $Z$ given $X_1,\ldots,X_p$, 
\item the limit has a natural interpretation as a nonlinear generalization of the familiar partial $R^2$ statistic for measuring the conditional dependence of $Y$ and $Z$ given $X_1,\ldots,X_p$, and 
\item all of the above hold under absolutely no assumptions on the laws of the random variables. 
\end{enumerate} 
The second contribution of this paper is a new variable selection algorithm based on the above measure of conditional dependence, called {\it Feature Ordering by Conditional Independence (FOCI)}, which is model-free, has no tuning parameters, and is provably consistent under sparsity assumptions. More importantly, it appears to perform very well in simulated and real datasets. The development of FOCI and the proof of its consistency are the major new contributions of this paper over \cite{chatterjee19}. It is not possible to devise such an algorithm using the univariate coefficient from \cite{chatterjee19}. 

The paper is organized as follows. The definition and properties of our coefficient are presented in Section \ref{resultsec}. Section \ref{intsec} discusses how to interpret the coefficient as a nonlinear generalization of partial $R^2$. A theorem about its rate of convergence is presented in Section~\ref{ratesec}. Our variable selection method is introduced in Section \ref{focisec} and a theorem about its consistency is stated in Section \ref{consistencysec}. The special case of linear regression with Gaussian predictors is illustrated in Section \ref{gaussiansec}. Applications to simulated and real datasets are presented in Section~\ref{examplesec}. The remaining sections are devoted to proofs. 


\section{The coefficient}\label{resultsec}
Let $Y$ be a random variable and $\bbx = (X_1,\ldots,X_p)$ and $\bbz = (Z_1,\ldots, Z_q)$ be random vectors, all defined on the same probability space. Here $q\ge 1$ and $p\ge 0$. The value $p=0$ means that $\bbx$ has no components at all.  Let $\mu$ be the law of $Y$. We propose the following quantity as a measure of the degree of conditional dependence of $Y$ and $\bbz$ given $\bbx$:
\begin{equation}\label{tdef0}
T = T(Y,\bbz|\bbx) := \frac{\int \ee(\var(\pp(Y\ge t|\bbz, \bbx)|\bbx)) d\mu(t)}{\int \ee(\var(1_{\{Y\ge t\}}|\bbx))d\mu(t)}.
\end{equation}
In the denominator, $1_{\{Y\ge t\}}$ is the indicator of the event $\{Y\ge t\}$. If the denominator equals zero, $T$ is undefined. (We will see below that this happens if and only if $Y$ is almost surely equal to a measurable function of $\bbx$, which is a degenerate case that we will ignore.) If $p=0$, then $\bbx$ has no components, and the conditional expectations and variances given $\bbx$ should be interpreted as unconditional expectations and variances. In this case we will  write $T(Y,\bbz)$ instead of  $T(Y,\bbz|\bbx)$.

Although the statistic $T$ has a somewhat complicated looking expression, it has a natural interpretation as a nonlinear generalization of the partial $R^2$ statistic for measuring the proportion of variation in $Y$ that is explained by $(\bbz,\bbx)$ but cannot be explained solely by $\bbx$. This is discussed in the next section. Specifically, see equation \eqref{inteq}. 

Note that $T$ is a non-random quantity that depends only the joint law of $(Y, \bbx,\bbz)$. Before stating our theorem about $T$, let us first see why $T$ is a reasonable measure of conditional dependence. Since taking conditional expectation decreases variance, we have that for any $t$, 
\[
\var(1_{\{Y\ge t\}}|\bbx) \ge \var(\pp(Y\ge t|\bbz, \bbx)|\bbx).
\]
This shows that the numerator in \eqref{tdef0} is less than or equal to the denominator, and so $T$ is always between $0$ and $1$. Now, if $Y$ and $\bbz$ are conditionally independent given $\bbx$, then $\pp(Y\ge t|\bbz, \bbx)$ is a function of $\bbx$ only, and hence $\var(\pp(Y\ge t|\bbz, \bbx)|\bbx) = 0$. Therefore in this situation, $T=0$. We will show later that the converse is also true. On the other hand, if $Y$ is almost surely equal to a measurable function of $\bbz$ given $\bbx$, then $\pp(Y\ge t|\bbz, \bbx) = 1_{\{Y\ge t\}}$ for any $t$. Therefore in this case, $T=1$. Again, we will prove later that the converse is  true. The following theorem summarizes these properties of~$T$.
\begin{thm}\label{popthm}
Suppose that $Y$ is not almost surely equal to a measurable function of $\bbx$ (when $p=0$, this means that $Y$ is not almost surely a constant). Then $T$ is well-defined and $0\le T \le 1$. Moreover, $T= 0$ if and only if $Y$ and $\bbz$ are conditionally independent given $\bbx$, and $T=1$ if and only if $Y$  is almost surely equal to a measurable function of $\bbz$ given $\bbx$. When $p=0$, conditional independence given $\bbx$ simply means unconditional independence. 
\end{thm}
The statistic $T$ is a generalization of a similar univariate measure defined in \cite{dss13, chatterjee19}. Having defined $T$, the main question is whether $T$ can be efficiently estimated from data. We will now present a consistent estimator of $T$, which is our conditional dependence coefficient. This generalizes a similar univariate estimator defined in \cite{chatterjee19}. Our data consists of $n$ i.i.d.~copies $(Y_1,\bbx_1,\bbz_1),\ldots,(Y_n, \bbx_n,\bbz_n)$ of the triple $(Y,\bbx,\bbz)$, where $n\ge 2$. For each $i$, let $N(i)$ be the index $j$ such that $\bbx_j$ is the nearest neighbor of $\bbx_i$ with respect to the Euclidean metric on $\rr^p$, where ties are broken uniformly at random. Let $M(i)$ be the index $j$ such that $(\bbx_j, \bbz_j)$ is the nearest neighbor of $(\bbx_i, \bbz_i)$ in $\rr^{p+q}$, again with ties broken uniformly at random. Let $R_i$ be the rank of $Y_i$, that is, the number of $j$ such that $Y_j\le Y_i$. If $p\ge 1$, our estimate of $T$ is 
\[
T_n = T_n(Y, \bbz|\bbx) := \frac{\sum_{i=1}^n (\min\{R_i, R_{M(i)}\} - \min\{R_i, R_{N(i)}\})}{\sum_{i=1}^n (R_i - \min\{R_i, R_{N(i)}\})}.
\]
If $p=0$, let $L_i$ be the number of $j$ such that $Y_j\ge Y_i$, let $M(i)$ denote the $j$ such that $\bbz_j$ is the nearest neighbor of $\bbz_i$ (ties broken uniformly at random),  and let
\[
T_n = T_n(Y, \bbz) := \frac{\sum_{i=1}^n (n\min\{R_i, R_{M(i)}\} - L_i^2)}{\sum_{i=1}^n L_i(n-L_i)}.
\]
In both cases, $T_n$ is undefined if the denominator is zero. 
The following theorem proves that $T_n$ is indeed a consistent estimator of $T$.
\begin{thm}\label{sampthm}
Suppose that $Y$ is not almost surely equal to a measurable function of $\bbx$. Then as  $n\to\infty$, $T_n \to T$ almost surely.
\end{thm}
\noindent {\it Remarks.} (1) If $p$ and $q$ are fixed, the statistic $T_n$ can be computed in $O(n\log n)$ time because nearest neighbors can be determined in $O(n\log n)$ time \cite{fbf77} and ranks can also be calculated in $O(n\log n)$ time \cite{knuth}. 

(2) No assumptions on the joint law of $(Y,\bbx,\bbz)$ are needed other than the non-degeneracy condition that $Y$ is not almost surely equal to a measurable function of $\bbx$. This condition is inevitable, because if this does not hold, then given $\bbx$, $Y$ is a constant; in this circumstance, $Y$ is both a function of $\bbz$ given $\bbx$ {\it and} independent of $\bbz$ given $\bbx$, and so there can be no reasonable measure of the degree of conditional dependence of $Y$ and $\bbz$ given $\bbx$.


(3) Although the limit of $T_n$ is guaranteed to be in $[0,1]$, the actual value of $T_n$ for finite $n$ may lie outside this interval.

(4) It is not easy to explain why $T_n$ is a consistent estimator of $T$ without going into the details of the proof, so we will not make that attempt here.

(5) We have not given a name to $T_n$, but if an acronym is desired for easy reference, one may call it CODEC, which is an acronym for Conditional Dependence Coefficient. In fact, this is the acronym that we use in the R code for computing $T_n$. 


(6) We have prepared an R package, called FOCI, that has a function for computing $T_n$ and a function for executing the variable selection algorithm FOCI presented in Section \ref{focisec} below. The package is available for download on CRAN~\cite{acpackage}.

(7) Besides variable selection, another natural area of applications of our coefficient is graphical models. This is currently under investigation.

(8) The consistency of $T_n$ raises the possibility of constructing a consistent test for conditional independence based on $T_n$. However, it is known that this is an impossible task, even for a single alternative hypothesis, if we demand that the level of the test be asymptotically uniformly bounded by some given $\alpha$ over the whole null hypothesis space~\cite{shahpeters18}. This is why the problem of nonparametric conditional independence testing for continuous random variables  is essentially unsolvable unless one is willing to  impose unverifiable assumptions. This contrasts starkly with the problem of nonparametric testing of {\it unconditional} independence, for which there are many useful and popular methods (see \cite{chatterjee19} for a survey).

(9) In view of Theorem \ref{sampthm}, it is natural to be curious about the rate of convergence of $T_n$ to $T$. This is investigated in Section \ref{ratesec}.

\section{Interpreting the coefficient}\label{intsec}
To interpret $T(Y,\bbz|\bbx)$, it is instructive to first consider the case of binary $Y$. Suppose that $Y$ is $\{0,1\}$-valued. Then $\mu$ is supported on $\{0,1\}$. Since $Y\ge 0$ always, we have $\pp(Y\ge 0|\bbz,\bbx) = 1_{\{Y\ge 0\}}=1$ always. Thus,
\[
\var(\pp(Y\ge 0|\bbz,\bbx)|\bbx) = \var(1_{\{Y\ge 0\}}|\bbx) = 0.
\]
Moreover, $Y = 1_{\{Y\ge 1\}}$. Combining all of this, we get that for binary $Y$, 
\begin{align*}
T(Y,\bbz|\bbx) &= \frac{\ee(\var(\ee(Y|\bbz,\bbx)|\bbx))}{\ee(\var(Y|\bbx))}.
\end{align*}
But, by the law of total variance, 
\[
\var(Y|\bbx) = \ee(\var(Y|\bbz,\bbx)|\bbx) + \var(\ee(Y|\bbz,\bbx)|\bbx).
\]
Thus, for binary $Y$, 
\begin{align*}
T(Y,\bbz|\bbx) &= 1 - \frac{ \ee(\var(Y|\bbz,\bbx))}{\ee(\var(Y|\bbx))}. 
\end{align*}
But this is just the {\it partial $R^2$}  that measures the proportion of variation in $Y$ that is explained by $(\bbz,\bbx)$ but cannot be explained solely by $\bbx$. Therefore, when $Y$ is binary, we have the identity
\[
T(Y,\bbz|\bbx) = R_{Y,\bbz|\bbx}^2. 
\]
For a general $Y$, let $Y_t := 1_{\{Y\ge t\}}$ for each $t$.  Then by the same calculation as above, we get
\begin{align*}
T(Y,\bbz|\bbx) &= 1 - \frac{\int \ee(\var(Y_t|\bbz,\bbx))d\mu(t)}{\int\ee(\var(Y_t|\bbx))d\mu(t)}.
\end{align*}
Let us now define a probability measure $\nu$ on $\rr$, which has density proportional to $\ee(\var(Y_t|\bbx))$ with respect to $\mu$. Then the above formula can be rewritten as 
\begin{align}\label{inteq}
T(Y,\bbz|\bbx) = \int R_{Y_t,\bbz|\bbx}^2 d\nu(t).
\end{align}
Thus, $T(Y,\bbz|\bbx)$ is a weighted average of $R_{Y_t,\bbz|\bbx}^2$ over all $t\in \rr$. But the random variable $Y$ is a linear combination of the binary variables $\{Y_t\}_{t\in \rr}$. Therefore $T(Y,\bbz|\bbx)$ is {\it also a measure of the proportion of variation in $Y$ that is explained by $(\bbz,\bbx)$ but cannot be explained solely by $\bbx$}. It generalizes the usual partial $R^2$ statistic $R^2_{Y,\bbz|\bbx}$ in a nonlinear way --- by breaking up $Y$ as a linear combination of binary variables, computing the partial $R^2$ for each binary variable, and combining these partial $R^2$ statistics by taking a weighted average. 

\section{Rate of convergence}\label{ratesec}
Suppose that $p\ge 1$, so that $\bbx$ has at least one component. (Recall that $q$ is always at least $1$.) To obtain a rate of convergence of $T_n$ to $T$, we need to make some assumptions about the distribution of $(Y,\bbx,\bbz)$, because otherwise, we believe that the convergence may be arbitrarily slow. The main issue is that we need some kind of control on the sensitivity of the conditional distribution of $Y$ given $\bbx$ and $\bbz$ on the values of $\bbx$ and $\bbz$. This is handled by the first assumption below. The second assumption is a matter of technical convenience.
\begin{itemize}
\item[(A1)] There are nonnegative real numbers $\beta$ and $C$ such that for any $t\in \rr$, $\bx,\bx'\in \rr^p$ and $\bz,\bz' \in \rr^q$, 
\begin{align*}
&|P(Y\ge t|\bbx = \bx, \bbz=\bz)  - P(Y\ge t|\bbx = \bx', \bbz=\bz')|\notag\\
&\le C(1+\|\bx\|^\beta +\|\bx'\|^\beta +\|\bz\|^\beta +\|\bz'\|^\beta )(\|\bx-\bx'\| + \|\bz-\bz'\|),
\end{align*}
and
\begin{align*}
&|P(Y\ge t|\bbx = \bx)  - P(Y\ge t|\bbx = \bx')|\\
&\le C(1+\|\bx\|^\beta +\|\bx'\|^\beta )\|\bx-\bx'\|.
\end{align*}
\item[(A2)] There are positive numbers $C_1$ and $C_2$ such that for any $t>0$, $\pp(\|\bbx\|\ge t)$ and $\pp(\|\bbz\|\ge t)$ are bounded by $C_1 e^{-C_2t}$. 
\end{itemize}
Note that assumption (A1) means that the conditional distribution of $Y$ given $(\bbx,\bbz)=(\bx,\bz)$ is a locally Lipschitz function of $(\bx,\bz)$, where the Lipschitz constant is allowed to grow at most polynomially in $\|\bx\|$ and $\|\bz\|$. Local Lipschitzness is a fairly relaxed assumption. It only excludes  esoteric cases where the conditional distribution of $Y$ given $(\bbx,\bbz)=(\bx,\bz)$ is a very rough function of $(\bx,\bz)$ (for example, like a Brownian path), which do not arise in any model used in practice.

Under the above assumptions, the following theorem shows that $T_n$ converges to $T$ essentially at the rate $n^{-1/(p+q)}$, up to an extra logarithmic term. 
\begin{thm}\label{ratethm}
Suppose that $p\ge 1$ and $q\ge 1$, and that  the assumptions \textup{(A1)} and \textup{(A2)} hold with some $\beta$ and $C$. Then, as $n\to\infty$, 
\[
T_n - T = O_P\biggl(\frac{(\log n)^{p+q+\beta+1}}{n^{1/(p+q)}}\biggr).
\]
\end{thm}
We believe that the rate $n^{-1/(p+q)}$ in Theorem \ref{ratethm} is the true rate of convergence of $T_n$ to $T$ when the variables are continuous.  It is not clear if there is some other statistic with the same properties as $T_n$ but with a better rate of convergence. 

Note that the case $p=0$ is not covered by Theorem \ref{ratethm}. This is the case where $T_n$ is a measure of unconditional, rather than conditional, dependence. When $p=0$ and $q=1$, we conjecture that $T_n-T = O_P(n^{-1/2})$. Moreover, under independence, we conjecture that $\sqrt{n}T_n$ obeys a central limit theorem when $p=0$ and $q=1$. At this moment, we do not know how to prove these conjectures. 


Conditions (A1) and (A2) are trivially satisfied if the support of $(Y,\bbx,\bbz)$ is a finite set, by choosing $\beta=0$ and a suitably large $C$. Another situation where it is easy to see that (A1) and (A2) hold is when $(Y,\bbx,\bbz)$ is normal, because then the conditional distribution of $Y$ given $(\bbx,\bbz)$ is again normal with a mean that is a linear function of $\bbx$ and $\bbz$, and a variance that does not depend on $\bbx$ and $\bbz$.

 More generally, the following result shows that (A1) is satisfied for a large class of densities with certain regularity and decay properties (and (A2) holds widely anyway). 
\begin{prop}\label{conditionprop}
Let $f(y|\bx)$ be the conditional probability density function of $Y$ given $\bbx=\bx$, assuming it exists. Suppose that $f$ is nonzero everywhere and differentiable with respect to $\bx$, and for each $i$, the function
\[
\biggl|\fpar{}{x_i}\log f(y|\bx)\biggr|
\] 
is bounded above by a polynomial in $|y|$ and $\|\bx\|$. Next, suppose that for any compact set $K\subseteq \rr^p$, the function $g(y) := \max_{\bx\in K} f(y|\bx)$ is bounded and decays faster than any negative power of $|y|$ as $|y|\to \infty$. Lastly, assume that for any $k\ge 1$, $\ee(Y^{2k}|\bbx=\bx)$ is bounded above by a polynomial in $\|\bx\|$. Then the second inequality in assumption \textup{(A1)} holds for some $C$ and $\beta$. A similar set of conditions on the conditional density of $Y$ given $\bbx=\bx$ and $\bbz=\bz$ ensures that the first inequality in \textup{(A1)} holds. 
\end{prop}


\section{Feature Ordering by Conditional Independence (FOCI)}\label{focisec}
In this section we propose a new variable selection algorithm for multivariate regression using a forward stepwise algorithm based on our measure of conditional dependence. The commonly used variable selection methods in the statistics literature use linear or  additive models. This includes classical methods~\cite{breiman95, gm93,chendonoho94, tibs96, efron04, friedman91, htf01, miller02} as well as modern ones~\cite{candestao07,   zou06, zouhastie05, yuanlin06, fanli01,  ravi09}. These methods are powerful and widely used in practice. However, they sometimes run  into problems when significant interaction effects or nonlinearities are present. We will later show an  example where methods based on linear and additive models fail to select any of the relevant predictors, {\it even in the complete absence of noise}.

Such problems can sometimes be overcome by model-free methods~\cite{Candes18, ho98, amitgeman97, breiman96, fs96, htf01, bfos84, battiti94, vergaraestevez14, breiman01}. These, too, are powerful and widely used techniques, and they perform better than model-based methods if interactions are present. On the flip side, their theoretical foundations are usually weaker than those of model-based methods.

The method that we are going to propose below, called {\it Feature Ordering by Conditional Independence~(FOCI)}, attempts to combine the best of both worlds by being fully model-free, as well as having a proof of consistency  under a set of assumptions.


The method is as follows. Let $Y$ be the response variable and let $\bbx = (X_j)_{1\le j\le p}$ be the set of predictors. The data consists of $n$ i.i.d.~copies of $(Y,\bbx)$. First, choose $j_1$ to be the index $j$ that maximizes $T_n(Y, X_j)$. Having obtained $j_1,\ldots, j_k$,  choose $j_{k+1}$ to be the index $j\notin\{j_1,\ldots, j_{k}\}$ that maximizes $T_n(Y, X_j|X_{j_1},\ldots,X_{j_{k}})$. Continue like this until arriving at the first $k$ such that $T_n(Y, X_{j_{k+1}}|X_{j_1},\ldots,X_{j_{k}})\le 0$, and then declare the chosen subset to be $\hat{S} := \{j_1,\ldots,j_k\}$. If there is no such $k$, define $\hat{S}$ to be the whole set of variables. It may also happen that $T_n(Y, X_{j_1})\le 0$. In that case declare $\hat{S}$ to be empty.



Although it is not required theoretically, we recommend that the predictor variables be standardized before running the algorithm. We will see later that FOCI performs well in examples, even if the true dependence of $Y$ on $\bbx$ is nonlinear in a complicated way.  In the next section we prove the consistency of FOCI under a set of assumptions on the law of $(Y,\bbx)$. 


If computational time is not an issue, one can try to add $m\ge 2$ variables at each step instead of just one. Although we do not explore this idea in this paper, it is possible that this gives improved results in certain situations. Similarly, one can try a forward-backward version of FOCI, analogous to the forward-backward version of ordinary stepwise selection. 

One can also consider implementing a forward stepwise algorithm like FOCI with other measures of conditional dependence. To the best of our knowledge, that has not yet been done. The closest cousin in the literature is an algorithm based on mutual information~\cite{battiti94}, but unlike FOCI, it does not have a well-defined stopping rule.

One deficiency of FOCI is that it only selects a subset of predictors, without actually fitting a predictive model. Following a suggestion from Rob Tibshirani, we recommend doing the following: First select a subset using FOCI, and then use random forests~\cite{breiman01} to fit a predictive model with the selected variables. This has two advantages over simply fitting random forests with the full set of predictors: (1) It picks out a small set of `important' variables, which may be useful for various reasons, and (2) it is computationally much less expensive. We saw that in real datasets, the prediction error of FOCI followed by random forests is only slightly worse than fitting random forests with the full set of predictors. On the other hand, the number of variables selected by FOCI is usually very small compared to the total number of variables. Some examples are given in Section \ref{examplesec}.

The stopping rule for FOCI may not be the best one. One can think of various other ways of using our measure of conditional (or any other) for variable selection. Our stopping rule seems to work well in practice, and we are able to prove consistency of variable selection for this rule. It is possible that there are other, better rules, which are also provably consistent. This merits further investigation.  

\section{Consistency of FOCI}\label{consistencysec}
Let $(Y,\bbx)$ be as in the previous section. For any subset of indices $S\subseteq \{1,\ldots, p\}$, let $\bbx_S:= (X_j)_{j\in S}$, and let $S^c := \{1,\ldots,p\}\setminus S$. In the machine learning literature, a subset $S$ is sometimes called {\it sufficient} \cite{vergaraestevez14} if $Y$ and $\bbx_{S^c}$ are conditionally independent given $\bbx_S$. This includes the possibility that  $S$ is the empty set, when it simply means that $Y$ and $\bbx$ are independent. Sufficient subsets are known as {\it Markov blankets} in the literature on graphical models~\cite[Section 3.2.1]{pearl88}, and are closely related to the concept of {\it sufficient dimension reduction} in classical statistics~\cite{cook07, ac09, li91}. If we can find a small subset of predictors that is sufficient, then our job is done, because these predictors contain all the relevant predictive information about $Y$ among the given set of predictors, and the statistician can then fit a predictive model based on this small subset of predictors.

Define $Q(\emptyset) := 0$, and for any nonempty set $S\subseteq \{1,\ldots,p\}$, let 
\begin{equation}\label{qsdef}
Q(S) := \int \var(\pp(Y\ge t|\bbx_S)) d\mu(t),
\end{equation}
where $\mu$ is the law of $Y$. We will prove later (Lemma \ref{monotonethm}) that $Q(S') \ge Q(S)$ whenever $S'\supseteq S$, with equality if and only $Y$ and $\bbx_{S'\setminus S}$ are conditionally independent given $\bbx_S$.    Thus if $S'\supseteq S$, the difference $Q(S')-Q(S)$ is a measure of how much extra predictive power is added by appending $\bbx_{S'\setminus S}$ to the set of predictors $\bbx_S$. 

Let $\delta$ be the largest number such that for any {\it insufficient} subset $S$, there is some $j\notin S$ such that $Q(S\cup \{j\})\ge Q(S)+\delta$. In other words, if $S$ is insufficient, there exists some index $j\notin S$ such that appending $X_j$ to $\bbx_S$ increases the predictive power by at least $\delta$. The main result of this section, stated below,  says that if $\delta$ is not too close to zero, then under some regularity assumptions on the law of $(Y,\bbx)$, the subset selected by FOCI is  sufficient with high probability. Note that a sparsity assumption is hidden in the condition that $\delta$ is not very small, because the definition of $\delta$ ensures that there is at least one sufficient subset of size $\le 1/\delta$. An interpretation of $\delta$ in the familiar setting of linear regression with Gaussian predictors is discussed in the next section. 

To prove our result, we  need the following two technical assumptions on the joint distribution of $(Y,\bbx)$. They are generalizations of the assumptions (A1) and (A2) from Section \ref{ratesec}. 
\begin{itemize}
\item[(A1$'$)] There are nonnegative real numbers $\beta$ and $C$ such that for any set $S\subseteq \{1,\ldots,p\}$ of size $\le 1/\delta+2$, any $\bx,\bx'\in \rr^S$ and any  $t\in \rr$,
\begin{align*}
&|P(Y\ge t|\bbx_S = \bx)  - P(Y\ge t|\bbx_S = \bx')|\notag\\
&\le C(1+\|\bx\|^\beta+\|\bx'\|^\beta) \|\bx-\bx'\|.
\end{align*}
\item[(A2$'$)] There are positive numbers $C_1$ and $C_2$ such that for any $S$ of size $\le 1/\delta+2$ and any $t>0$, $\pp(\|\bbx_S\|\ge t)\le C_1 e^{-C_2t}$. 
\end{itemize}
Proposition \ref{conditionprop} shows that the above assumptions are satisfied in a wide variety of situations. The following theorem shows that under the above assumptions, the subset chosen by FOCI is sufficient with high probability.
\begin{thm}\label{selectthm}
Suppose that $\delta >0$, and that the assumptions \textup{(A1$'$)} and \textup{(A2$'$)} hold. Let $\hat{S}$ be the subset selected by FOCI with a sample of size $n$. There are positive real numbers $L_1$, $L_2$  and $L_3$ depending only on $C$, $\beta$, $C_1$, $C_2$ and $\delta$ such that $\pp(\hat{S} \textup{ is sufficient})\ge 1- L_1p^{L_2}e^{-L_3n}$.
\end{thm}
The main implication of Theorem \ref{selectthm} is that if $\delta$ is not too close to zero, and $n\gg \log p$, then with high probability, FOCI chooses a sufficient set of predictors. In particular, this theorem allows $p$ to be quite large compared to $n$, as long as $\delta$ is not too small. 


Although Theorem \ref{selectthm} works under the assumption that $\delta$ is fixed (because the constants $L_1$, $L_2$ and $L_3$ depend on $\delta$ in an unspecified manner), it is possible that a deeper analysis can allow us to take $\delta \to 0$ as $n\to\infty$. For that, the dependences of $L_1$, $L_2$ and $L_3$ on $\delta$ will have to be made explicit. It is possible that such an improvement can be made by carefully reworking the steps in the proof, to at least get a consistency result when $\delta \to 0$ slower than $(\log n)^{-1}$. To get anything better than that will probably require entirely new ideas. 

Theorem \ref{selectthm} gives conditions under which the subset selected by FOCI is sufficient with high probability. This is intended to be useful for practitioners, by giving them confidence that the selected subset is indeed sufficient. In practice, as we will see in Section \ref{examplesec}, the subsets selected by FOCI are quite small. However, Theorem \ref{selectthm} does not guarantee the smallness of the subset size. It would be desirable to have an improved version of Theorem~\ref{selectthm}, which not only guarantees that $\hat{S}$ is sufficient with high probability, but also that $|\hat{S}|$ is small with high probability. As of now, we do not know how to prove such a result.

\section{Interpreting $\delta$}\label{gaussiansec}
In this section we will try to understand the meaning of the quantity $\delta$ defined in the previous section, in the familiar context of linear regression with normally distributed predictor variables. Suppose that $\bbx$ is a normal random vector with zero mean and arbitrary covariance structure, and that
\[
Y = \beta\cdot \bbx + \ve,
\]
where $\beta\in \rr^p$ is a  vector of coefficients and $\ve\sim N(0,\sigma^2)$ is independent of $\bbx$, with nonzero $\sigma$. Then $Y$ is also a normal random variable with mean zero. Let $\tau^2 := \var(Y)$. Let $\delta$ be the quantity defined in the previous section, for this $Y$ and $\bbx$. 

For any nonempty $S\subsetneq \{1,\ldots, p\}$ and any $j\in \{1,\ldots,p\}\setminus S$, let $\rho(S,j)$ be the partial $R^2$ of $Y$ and $X_j$ given $\bbx_S$. Let $\rho(\emptyset, j)$ be the usual $R^2$ (that is, squared correlation) between $Y$ and $X_j$. 

Note that if $S$ is a sufficient set of predictors, then $\rho(S, j)=0$ for any $j\notin S$. Conversely, if $\rho(S,j)=0$ for all $j\notin S$, then by normality, $Y$ and $\bbx_{S^c}$ are conditionally independent given $\bbx_S$, and hence $S$ is sufficient. Thus, $S$ is sufficient if and only if $\rho(S,j)=0$ for all $j\notin S$. So if $S$ is insufficient, then there is at least one $j\notin S$ such that $\rho(S,j)>0$. 

Let $\delta'$ be the largest number such that for any insufficient set $S$, there is some $j\notin S$ such that  $\rho(S,j)\ge \delta'$. The following result shows that $\delta'$ is comparable to $\delta$, up to constant multiples depending only on $\sigma$ and $\tau$. 
\begin{thm}\label{deltacompare}
Let all notations be as above. There are positive universal constants $C_1$ and $C_2$ such that 
\[
\frac{C_1\sigma^2}{\tau^2} \delta' \le \delta \le \frac{C_2 \tau^2}{\sigma^2} \delta'.
\]
\end{thm}
Thus, in the Gaussian setup, $\delta$ is equivalent to the analogous quantity computed using the usual partial $R^2$ instead of our measure of conditional dependence. The above result is proved in Section \ref{gaussianproof}.

\section{Examples}\label{examplesec}
In this section we present some applications of our methods to simulated examples and  real datasets. In all examples, the covariates were standardized prior to the analysis.
\begin{ex}\label{sumUnifEx}
Let $X_1$ and $X_2$ be independent Uniform$[0, 1]$ random variables, and define 
\[
Y := X_1 + X_2\pmod 1.
\]
The relationship between $Y$ and $(X_1,X_2)$ has three main features:
\begin{enumerate}
\item $Y$ is a function of $(X_1,X_2)$, 
\item unconditionally, $Y$ is independent of $X_2$, and 
\item conditional on $X_1$, $Y$ is a function of $X_2$. 
\end{enumerate}
Let $n=1000$. In about $95$ percent of our simulations, $T_n(Y, (X_1,X_2))$ took values between $0.88$ and $0.94$,  $T_n(Y,X_2|X_1)$ was between $0.88$ and $0.94$, and $T_n(Y,X_2)$ was between $-0.07$ and $0.07$, in agreement with the above properties. Other measures of conditional dependence, such as conditional distance correlation~\cite{Wang15}, were unable to gauge the strength of the conditional dependency between $Y$ and $X_2$ given $X_1$.
\end{ex}
\begin{ex}
Let $X_1$ and $X_2$ be independent~$N(0,1)$ random variables, and define 
\[
Y := X_1^2 + X_2^2, \ \ \ Z:= \arctan(X_1/X_2).
\]
Then unconditionally, $Y$ is independent of $Z$, and conditional on $X_1$, $Y$ is a function of $Z$. Let $n=1000$. In about $95$ percent of our simulations, $T_n(Y, Z)$ took values between $-0.06$ and $0.05$,  and $T_n(Y,Z|X_1)$ was between $0.79$ and $0.84$, in agreement with the above properties. Again, other measures of conditional dependence were unable to capture the strength of the conditional dependence between $Y$ and $Z$ given $X_1$.

\end{ex}

\begin{ex}\label{simex}
Let $X_1,\ldots, X_{1000}$ be independent~$N(0,1)$ random variables and let 
\begin{equation*}
Y = X_1X_2 + \sin(X_1X_3).
\end{equation*}
With a sample of size $2000$ from the above model, FOCI was able to select the correct subset $\{X_1,X_2,X_3\}$ more than $90$ percent of the time. On the other hand, popular variable selection algorithms based on linear models, such as ordinary forward stepwise, Lasso~\cite{tibs96}, the Dantzig selector~\cite{candestao07}, and SCAD~\cite{fanli01} were essentially never able to pick out the correct subset. (The tuning parameters for Lasso, Dantzig selector and SCAD were chosen using $10$-fold cross-validation, and the AIC criterion was used for stopping in forward stepwise.) Even methods based on nonlinear additive models, such as SPAM~\cite{ravi09}, were generally unable to find the correct subset. The only other methods that successfully detected the importance of $X_1$, $X_2$ and $X_3$ were random forests~\cite{breiman01} and mutual information~\cite{battiti94}, but the computational times for these methods were many times greater than that of FOCI. 
\end{ex}

\begin{ex}\label{simex2}
Again, let $X_1,\ldots, X_{1000}$ be independent~$N(0,1)$ random variables and let 
\begin{equation*}
Y = X_1X_2 + X_1 - X_3 + \varepsilon,
\end{equation*}
where $\varepsilon\sim N(0,1)$ is a noise term that is independent of $X_i$'s. With a sample of size $2000$ from this model, FOCI was able to select the correct subset $\{X_1,X_2,X_3\}$ in $99.5$ percent of simulations. Methods based on linear models were generally able to pick out $X_1$ and $X_3$ but almost never detected the role of $X_2$. SPAM was able to pick out all three variables in about a quarter of the simulations. Again, the only other methods that successfully detected the importance of $X_1$, $X_2$ and $X_3$ were random forests and mutual information, but at a far greater computational cost than FOCI. 
\end{ex}

\begin{ex}
We tried out FOCI on the following three benchmark real data examples, all from the UCI Machine Learning Repository~\cite{dg19}:
\begin{enumerate}
\item \textit{Spambase data}.  Consists of $4601$ observations, each corresponding to one email, and $57$ features for each observation. The response variable is binary, indicating whether the email is a spam email or not. 
\item \textit{Polish companies bankruptcy data}. Consists of $19967$ observations with $64$ features. Each sample corresponds to a company in Poland. The response variable is binary, indicating whether or not the company was bankrupted after a period of time. 
\item \textit{Million song data}. Consists of $515345$ observations with $90$ features. Each sample corresponds to the audio features of a song published sometime ranging from 1922 to 2011. The response variable is the year that the song was published. 
\end{enumerate}
FOCI was compared with forward stepwise, Lasso, Dantzig selector and SCAD. For each method, after selecting the variables, a predictive model was fitted to a training set using random forests. As before, the tuning parameters for Lasso, Dantzig selector and SCAD were chosen using $10$-fold cross-validation, and the AIC criterion was used for stopping in forward stepwise. Mean squared prediction errors (MSPE) were estimated using a test set. The sizes of the selected subsets and the MSPEs are reported in Table~\ref{newtable}. In all three examples, FOCI attained similar prediction errors as the other methods, but with a significantly fewer number of variables.
\end{ex}

\begin{table}[t]
	\renewcommand\arraystretch{1.2}
	\begin{center}
		\begin{scriptsize}
			\caption{Applications of FOCI to real data.\label{newtable}}
			\begin{tabular}{lcccccccc}
				\toprule
				& \multicolumn{2}{c}{Spambase data} & & \multicolumn{2}{c}{Polish companies data} & & \multicolumn{2}{c}{Million song data}\\
				\cmidrule{2-3} \cmidrule{5-6} \cmidrule{8-9}
				Method & Subset size& MSPE & & Subset size & MSPE & & Subset size & MSPE\\
			\midrule
			FOCI & 14 & 0.045 & & 5 & 0.022 & & 17 & 88.085\\
			Forward stepwise & 56 & 0.039 & & 62 & 0.020 && 90 & 87.226\\ 
			Lasso & 55 & 0.041  & & 48 & 0.021 & & 86 & 87.260\\
			Dantzig selector & 53 & 0.041 & & 7 & 0.023  & & 90 & 87.226\\
			SCAD & 38 & 0.041 & & 4 & 0.025 & & 85 & 87.319\\
			\bottomrule
			\end{tabular}
		\end{scriptsize}
	\end{center}
\end{table}

\begin{table}[t]
	\renewcommand\arraystretch{1.2}
	\begin{center}
		\begin{scriptsize}
			\caption{Comparison with random forests.\label{FOCIvsRandomForest}}
			\begin{tabular}{lccc}
				\toprule
				Dataset & FOCI subset size$/$Total set size & MSPE FOCI & MSPE random forest\\
				\midrule
				Spambase  & 14/57 & 0.045 & 0.040 \\
				Polish companies & 5/64 & 0.022 & 0.020 \\ 
				Million song & 17/90 & 88.085 & 87.260 \\ 
				\bottomrule
			\end{tabular}
		\end{scriptsize}
	\end{center}
\end{table}

\begin{ex}
In Section \ref{focisec}, we recommended fitting a predictive model using random forests with the set of variables selected by FOCI. To test the validity of this approach, we computed the prediction errors for random forests with the full set of predictors versus FOCI followed by random forests, in the three real datasets considered above. The results are displayed in Table \ref{FOCIvsRandomForest}. We see that FOCI followed by random forests attains almost the same MSPE as random forests with the full set of variables; but in each case, the number of variables selected by FOCI is small compared to the total number of variables.
\end{ex}

\section{Restatement of Theorems \ref{popthm} and \ref{sampthm}}\label{resultsec2}
Beginning with this section, the rest of the paper is devoted to proofs. Throughout the rest of the manuscript, whenever we say that a random variable $Y$ is a function of another variable $X$, we will mean that $Y=f(X)$ almost surely for some measurable function $f$. 

First, we focus on Theorems \ref{popthm} and \ref{sampthm}. To prove these theorems, it is convenient to  break up the estimators into pieces. This gives certain `elaborate' versions of Theorems~\ref{popthm} and~\ref{sampthm}, which are interesting in their own right. First, suppose that $p\ge 1$. Define 
\begin{equation}\label{qndef1}
Q_n(Y, \bbz|\bbx) := \frac{1}{n^2}\sum_{i=1}^n (\min\{R_i, R_{M(i)}\} - \min\{R_i, R_{N(i)}\})
\end{equation}
and 
\begin{equation}\label{sndef1}
S_n(Y,\bbx) := \frac{1}{n^2}\sum_{i=1}^n (R_i - \min\{R_i, R_{N(i)}\}).
\end{equation}
Let $\mu$ denote the law of $Y$.  We will see later that the following theorem implies both Theorem \ref{popthm} and Theorem \ref{sampthm} in the case $p\ge 1$.
\begin{thm}\label{mainthm}
Suppose that $p\ge 1$. As $n\to \infty$, the statistics $Q_n(Y, \bbz|\bbx) $ and $S_n(Y,\bbx)$ converge almost surely to deterministic limits. Call these limit $a$ and $b$, respectively. Then
\begin{enumerate}
\item[\textup{(i)}] $0\le a\le b$.
\item[\textup{(ii)}] $Y$ is conditionally independent of $\bbz$ given $\bbx$ if and only if $a=0$.
\item[\textup{(iii)}] $Y$ is conditionally a function of $\bbz$ given $\bbx$ if and only if $a=b$.
\item[\textup{(iv)}] $Y$ is not a function of $\bbx$ if and only if $b>0$.
\end{enumerate}
Explicitly, the values of $a$ and $b$ are given by
\begin{align*}
a = \int \ee(\var(\pp(Y\ge t|\bbz, \bbx)|\bbx)) d\mu(t)
\end{align*}
and 
\begin{align*}
b &= \int \ee(\var(1_{\{Y\ge t\}}|\bbx))d\mu(t)\\
&= \int \ee(\pp(Y\ge t|\bbx)(1-\pp(Y\ge t|\bbx)))d\mu(t).
\end{align*}
\end{thm}
Next, suppose that $p=0$. Define
\begin{equation}\label{qndef2}
Q_n(Y,\bbz) := \frac{1}{n^2}\sum_{i=1}^n \biggl(\min\{R_i, R_{M(i)}\} - \frac{L_i^2}{n}\biggr)
\end{equation}
and 
\begin{equation}\label{sndef2}
S_n(Y)  := \frac{1}{n^3}\sum_{i=1}^n L_i(n-L_i).
\end{equation}
We will prove later that the following theorem implies Theorems \ref{popthm} and \ref{sampthm}  when $p=0$. 
\begin{thm}\label{basethm}
As $n\to \infty$, $Q_n(Y,\bbz)$ and $S_n(Y)$ converge almost surely to  deterministic limits $c$ and $d$, satisfying the following properties:
\begin{enumerate}
\item[\textup{(i)}] $0\le c\le d$.
\item[\textup{(ii)}] $Y$ is independent of $\bbz$ if and only if $c=0$.
\item[\textup{(iii)}] $Y$ is a function of $\bbz$ if and only if $c=d$.
\item[\textup{(iv)}] $d>0$ if and only if $Y$ not a constant. 
\end{enumerate}
Explicitly,
\[
c = \int \var(\pp(Y\ge t|\bbz)) d\mu(t),
\]
and
\begin{align*}
d &= \int \var(1_{\{Y\ge t\}}) d\mu(t)\\
&=\int \pp(Y\ge t)(1-\pp(Y\ge t)) d\mu(t).
\end{align*}
\end{thm}
It is not difficult to see that whenever $Y$ has a continuous distribution, $d=1/6$. In this case, there is no need for estimating $d$ using $S_n(Y,\bbz)$. On the other hand, the value of $d$ may be dependent on the distribution of $Y$ when the distribution is not continuous. In such cases, $d$ needs to be estimated from the data using $S_n(Y,\bbz)$.

\section{Proofs of  Theorems \ref{popthm} and \ref{sampthm} using Theorems \ref{mainthm} and \ref{basethm}}\label{mainpfsec}
Suppose that $p\ge 1$. Recall the quantities $a$ and $b$ from the statement of Theorem \ref{mainthm}, and notice that $T=a/b$. Suppose that $Y$ is not a function of $\bbx$. Then by conclusion (iv) of Theorem \ref{mainthm}, $b>0$, and hence $T$ is well-defined. Moreover, conclusion (i) implies that $0\le T\le1$, conclusion (ii) implies that $T=0$ if and only if $Y$ and $\bbz$ are conditionally independent given $\bbx$, and conclusion (iii) implies that $Y$ is a function of $\bbz$ given $\bbx$ if and only if $T=1$. This proves Theorem \ref{popthm} when $p\ge 1$. Next, note that $T_n = Q_n/S_n$, where $Q_n = Q_n(Y,\bbz|\bbx)$ and $S_n = S_n(Y,\bbx)$, as defined in \eqref{qndef1} and \eqref{sndef1}. By Theorem~\ref{mainthm}, $Q_n \to a$ and $S_n \to b$ in probability. Thus, $T_n \to a/b =T$ in probability. This proves Theorem \ref{sampthm} when $p\ge 1$.

Next, suppose that $p=0$. The proof proceeds exactly as before, but using Theorem \ref{basethm}. Here $T=c/d$, where $c$ and $d$ are the quantities from Theorem~\ref{basethm}. Suppose that $Y$ is not a function of $\bbx$, which in this case just means that $Y$ is not a constant. Then by conclusion (iv) of Theorem~\ref{basethm}, $d>0$, and hence $T$ is well-defined. Moreover, conclusion (i) implies that $0\le T\le1$, conclusion (ii) implies that $T=0$ if and only if $Y$ and $\bbz$ are independent, and conclusion (iii) implies that $Y$ is a function of $\bbz$ if and only if $T=1$. This proves Theorem \ref{popthm} when $p=0$. Next, note that $T_n = Q_n/S_n$, where $Q_n = Q_n(Y,\bbz)$ and $S_n = S_n(Y)$, as defined in \eqref{qndef2} and \eqref{sndef2}. By Theorem~\ref{basethm}, $Q_n \to c$ and $S_n \to d$ in probability. Thus, $T_n \to c/d =T$ in probability. This proves Theorem \ref{sampthm} when $p=0$.

\section{Preparation for the proofs of Theorems \ref{mainthm} and \ref{basethm}}
In this section we prove some lemmas that are needed for the proofs of Theorems \ref{mainthm} and \ref{basethm}. Let $Y$ be a random variable and $\bbx$ be an $\rr^p$-valued random vector, defined on the same probability space. 
Define 
\[
F(t) :=  \pp(Y\le t), \ \ \ G(t):= \pp(Y\ge t).
\]
By the existence of regular conditional probabilities on regular Borel spaces (see for example \cite[Theorem 2.1.15 and Exercise 5.1.16]{durrett10}), for each Borel set $A\subseteq \rr$ there is a measurable map $\bx \mapsto \mu_\bx(A)$ from $\rr^p$ into $[0,1]$, such that 
\begin{enumerate}[(i)]
\item for any $A$, $\mu_\bbx(A)$ is a version of $\pp(Y\in A|\bbx)$, and 
\item with probability one, $\mu_\bbx$ is a probability measure on $\rr$. 
\end{enumerate}
In the above sense, $\mu_\bx$ is the conditional law of $Y$ given $\bbx=\bx$. For each $t$, let 
\[
F_\bbx(t) := \mu_\bbx((-\infty,t]), \ \ G_\bbx(t) := \mu_\bbx([t, \infty)). 
\]
Define
\begin{align}\label{qdef}
Q(Y,\bbx) := \int \var(G_\bbx(t)) d\mu(t). 
\end{align}
\begin{lmm}\label{indepthm1}
Let $Q(Y,\bbx)$ be as above. Then $Q(Y,\bbx) =0$ if and only if $Y$ and $\bbx$ are independent. 
\end{lmm}
\begin{proof}
If $Y$ and $\bbx$ are independent, then for any $t$, $\pp(Y\ge t|\bbx)=\pp(Y\ge t)$ almost surely. Thus, $G_\bbx(t) = G(t)$ almost surely, and so $\var(G_\bbx(t)) =0$. Consequently, $Q(Y,\bbx)=0$. 

Conversely, suppose that $Q(Y,\bbx)=0$. Then there is a set $A\subseteq \rr$ such that $\mu(A)=1$ and $\var(G_\bbx(t))=0$ for every $t\in A$. Since $\ee(G_\bbx(t))=G(t)$, $G_\bbx(t)=G(t)$ almost surely for each $t\in A$.  We claim that $A=\rr$. 

To show this, take any $t\in \rr$. If $\mu(\{t\})>0$, then clearly $t$ must be a member of $A$ and there is nothing more to prove. So assume that $\mu(\{t\})=0$. This implies that $G$ is right-continuous at $t$. 

There are two possibilities. First, suppose that $G(s)<G(t)$ for all $s>t$. Then for each $s>t$, $\mu([t,s)) >0$, and hence $A$ must intersect $[t,s)$. This shows that there is a sequence $r_n$ in $A$ such that $r_n$ decreases to $t$. Since $G_\bbx(r_n)=G(r_n)$ almost surely for each $n$, this implies that with probability one,
\[
G_\bbx(t) \ge \lim_{n\to \infty} G_\bbx(r_n)= \lim_{n\to\infty}G(r_n)=G(t).
\]
But $\ee(G_\bbx(t))=G(t)$. Thus, $G_\bbx(t) = G(t)$ almost surely. 

The second possibility is that there is some $s>t$ such that $G(s)=G(t)$. Take the largest such $s$, which exists because $G$ is left-continuous. If $s=\infty$, then $G(t)=G(s)=0$, and hence $G_\bbx(t)=0$ almost surely because $\ee(G_\bbx(t))=G(t)$. Suppose that $s<\infty$. Then either $\mu(\{s\})>0$, which implies that $G_\bbx(s)=G(s)$ almost surely, or $\mu(\{s\})=0$ and $G(r)<G(s)$ for all $r>s$, which again implies that $G_\bbx(s)=G(s)$ almost surely, by the previous paragraph. Therefore in either case, with probability one,
\[
G_\bbx(t)\ge G_\bbx(s)=G(s)=G(t).
\]
Since $\ee(G_\bbx(t))=G(t)$, this implies that $G_\bbx(t)=G(t)$ almost surely.

This completes the proof of our claim that $\var(G_\bbx(t))=0$ for every $t\in \rr$. In particular, for each $t\in \rr$, $G_\bbx(t)=G(t)$ almost surely. Therefore, for any $t\in \rr$ and any Borel set $B\subseteq \rr^p$,
\begin{align*}
\pp(\{Y\ge t\}\cap \{\bbx\in B\}) &= \ee(\pp(Y\ge t|\bbx) 1_{\{\bbx\in B\}})\\
&= G(t) \pp(\bbx\in B) = \pp(Y\ge t) \pp(\bbx\in B).
\end{align*} 
This proves that $Y$ and $\bbx$ are independent. 
\end{proof}
Let $\bbz$ be an $\rr^q$-valued random vector defined on the same probability space as $Y$ and $\bbx$, and let $\bbw = (\bbx,\bbz)$ be the concatenation of $\bbx$ and $\bbz$. 
\begin{lmm}\label{monotonethm}
Let $\bbw$ be as above. Then $Q(Y,\bbw)\ge Q(Y,\bbx)$, and equality holds if and only if $Y$ and $\bbz$ are conditionally independent given $\bbx$. 
\end{lmm}
\begin{proof}
Since $G_\bbx(t) = \ee(G_\bbw(t)|\bbx)$, it follows that for each $t$,
\[
\var(G_\bbx(t)) \le \var(G_\bbw(t)).
\]
Consequently, $Q(Y, \bbw)\ge Q(Y,\bbx)$. If $Y$ and $\bbz$ are conditionally independent given $\bbx$, then for any $t$,
\[
G_\bbw(t) = \pp(Y\ge t|\bbx, \bbz) = \pp(Y\ge t|\bbx) = G_\bbx(t). 
\]
Thus, $Q(Y,\bbw)=Q(Y,\bbx)$. 
Conversely, suppose that $Q(Y,\bbw)=Q(Y,\bbx)$. Notice that
\begin{align*}
\var(G_\bbw(t))-\var(G_\bbx(t)) &= \var(G_\bbw(t)) - \var(\ee(G_\bbw(t)|\bbx))\\
&= \ee(\var(G_\bbw(t)|\bbx))\\
&= \ee(G_\bbw(t)-G_\bbx(t))^2. 
\end{align*}
Thus,
\[
Q(Y,\bbw)-Q(Y,\bbx)=\int \ee(G_\bbw(t)-G_\bbx(t))^2d\mu(t). 
\]
So, if $Q(Y,\bbw)=Q(Y,\bbx)$, then there is a Borel set $A\subseteq \rr$ such that $\mu(A)=1$ and $G_\bbw(t)= G_\bbx(t)$ almost surely for every $t\in A$. We claim that $A=\rr$. Let us now prove this claim. The proof is similar to the proof of the analogous claim in Lemma \ref{indepthm1}, with a few additional complications.

Take any $t\in \rr$. If $\mu(\{t\})>0$, then clearly $t$ must be a member of $A$. So assume that $\mu(\{t\})=0$. As before, this implies that $G$ is right-continuous at $t$. Take any sequence $t_n$ decreasing to $t$. Then $G(t)-G(t_n)\to 0$. But 
\[
G(t)-G(t_n) = \ee(G_\bbx(t)-G_\bbx(t_n)),
\]
and $G_\bbx(t)-G_\bbx(t_n)$ is a nonnegative random variable. Thus, $G_\bbx(t)-G_\bbx(t_n)\to 0$ in probability, and therefore there is a subsequence $n_k$ such that $G_\bbx(t_{n_k})$ converges to $G_\bbx(t)$ almost surely. But from the properties of the regular conditional probability $\mu_\bx$ we know that $G_\bbx$ is a non-increasing function almost surely. Thus, it follows that $G_\bbx$ is right-continuous at $t$ almost surely.

Now, as before, there are two possibilities. First, suppose that $G(s)<G(t)$ for all $s>t$. Then for each $s>t$, $\mu([t,s)) >0$, and hence $A$ must intersect $[t,s)$. This shows that there is a sequence $r_n$ in $A$ such that $r_n$ decreases to $t$. Since $G_\bbw(r_n)=G_\bbx(r_n)$ almost surely for each $n$ and $G_\bbx$ is right-continuous at $t$ with probability one, this implies that with probability one,
\[
G_\bbw(t) \ge \lim_{n\to \infty} G_\bbw(r_n)= \lim_{n\to\infty}G_\bbx(r_n)=G_\bbx(t).
\]
But $\ee(G_\bbw(t)|\bbx)=G_\bbx(t)$. Thus, $G_\bbw(t) = G_\bbx(t)$ almost surely. 

The second possibility is that there is some $s>t$ such that $G(s)=G(t)$. Take the largest such $s$, which exists because $G$ is left-continuous. If $s=\infty$, then $G(t)=G(s)=0$, and hence $G_\bbw(t) = G_\bbx(t)=0$ almost surely because $\ee(G_\bbw(t)) = \ee(G_\bbx(t))=G(t)$. Suppose that $s<\infty$. Then either $\mu(\{s\})>0$, which implies that $G_\bbw(s)=G_\bbx(s)$ almost surely (by the previous step), or $\mu(\{s\})=0$ and $G(r)<G(s)$ for all $r>s$, which again implies that $G_\bbw(s)=G_\bbx(s)$ almost surely (also by the previous step). Therefore in either case, with probability one,
\[
G_\bbw(t)\ge G_\bbw(s)= G_\bbx(s).
\]
Now, $\pp(Y\in [t,s))=0$, and hence $\pp(Y\in [t,s)|\bbx)=0$ almost surely. In other words, $G_\bbx(t)=G_\bbx(s)$ almost surely. Thus, $G_\bbw(t)\ge G_\bbx(t)$ almost surely. 
Since $\ee(G_\bbw(t)|\bbx)=G_\bbx(t)$, this implies that $G_\bbw(t)=G_\bbx(t)$ almost surely. This completes the proof of our claim that $A=\rr$.

Therefore, for any $t\in \rr$ and any Borel set $B\subseteq \rr^{p'}$,
\begin{align*}
\pp(\{Y\ge t\}\cap \{\bbz\in B\}|\bbx) &= \ee(\pp(\{Y\ge t\}\cap \{\bbz\in B\}|\bbw)|\bbx)\\
&= \ee(\pp(Y\ge t|\bbw) 1_{\{\bbz\in B\}}|\bbx)\\
&= \ee(G_\bbx(t) 1_{\{\bbz\in B\}}|\bbx) \\
&= \pp(Y\ge t |\bbx) \pp(\bbz\in B|\bbx).
\end{align*} 
This proves that $Y$ and $\bbz$ are conditionally independent given $\bbx$. 
\end{proof}

Let $\bbx_1,\bbx_2,\ldots$ be an infinite sequence of i.i.d.~copies of $\bbx$. For each $n\ge 2$ and each $1\le i\le n$, let $\bbx_{n,i}$ be the Euclidean nearest-neighbor of $\bbx_i$ among $\{\bbx_j: 1\le j\le n, j\ne i\}$. Ties are broken at random. 
\begin{lmm}\label{nnlmm}
With probability one, $\bbx_{n,1}\to \bbx_1$ as $n\to \infty$.
\end{lmm}
\begin{proof}
 Let $\nu$ be the law of $\bbx$. Let $A$ be the support of $\nu$. Recall that $A$ is the set of all $\bx\in \rr^p$ such that any open ball containing $\bx$ has strictly positive $\nu$-measure. From this definition it follows easily that the complement of $A$ is a countable union of open balls of $\nu$-measure zero. Consequently, $\bbx\in A$ with probability one.
 
 Take any $\ve>0$. Let $B$ be the ball of radius $\ve$ centered at $\bbx_1$. Then 
\begin{align*}
\pp(\|\bbx_1-\bbx_{n,1}\|\ge \ve|\bbx_1) &\le (1-\nu(B))^{n-1}
\end{align*}
Since $\bbx_1\in A$ almost surely, it follows that $\nu(B)>0$ almost surely. Thus,
\begin{align*}
\lim_{n\to\infty} \pp(\|\bbx_1-\bbx_{n,1}\|\ge \ve|\bbx_1)=0
\end{align*}
almost surely, and hence
\[
\lim_{n\to\infty} \pp(\|\bbx_1-\bbx_{n,1}\|\ge \ve)=0.
\] 
This proves that $\|\bbx_1-\bbx_{n,1}\|\to 0$ in probability. But  $\|\bbx_1-\bbx_{n,1}\|$ is decreasing in $n$. Therefore $\|\bbx_1-\bbx_{n,1}\|\to 0$ almost surely.
\end{proof}
Take any particular realization of $\bbx_1,\ldots, \bbx_n$. In this realization, for each $1\le i\le n$, let $K_{n,i}$ be the number of $j$ such that $\bbx_i$ is a nearest neighbor of $\bbx_j$ (not necessarily the randomly chosen one) and $\bbx_j \ne \bbx_i$. The following is a well-known geometric fact (see for example \cite[page 102]{yukich98}). 
\begin{lmm}\label{geomlmm}
There is a deterministic constant $C(p)$, depending only on the dimension $p$, such that $K_{n,1}\le C(p)$ always. 
\end{lmm}
\begin{proof}
Consider a triangle with vertices $\bx$, $\by$ and $\bz$ in $\rr^p$, where $\by\ne \bx$ and $\bz\ne \bx$. Suppose that the angle at $\bx$ is strictly less than $60^\circ$ and $\|\bx-\by\|\le \|\bx-\bz\|$. Then 
\begin{align*}
\frac{(\by-\bx)\cdot(\bz-\bx)}{\|\by-\bx\|\|\bz-\bx\|}> \cos 60^\circ = \frac{1}{2}. 
\end{align*}
Consequently,
\begin{align*}
\|\bz-\by\|^2 &= \|\bz-\bx\|^2 + \|\bx-\by\|^2 + 2(\bz-\bx)\cdot(\bx-\by)\\
&< \|\bz-\bx\|^2 + \|\bx-\by\|^2 -\|\by-\bx\|\|\bz-\bx\|\\
&\le \|\bz-\bx\|^2,
\end{align*}
where the last inequality holds because $\|\bx-\by\|\le \|\bx-\bz\|$. Thus, if $K$ is a cone at $\bx$ of aperture less than $60^\circ$, and $\bx_1,\ldots,\bx_m$ is a finite list of points in $K\setminus\{\bx\}$ (not necessarily distinct), then there can be at most one $i$ such that the nearest neighbor of $\bx_i$ in $\{\bx, \bx_1,\ldots,\bx_m\}$ is $\bx$.

Now, it is not difficult to see that there is a deterministic constant $C(p)$ depending only on $p$ such that the whole of $\rr^p$ can be covered by at most $C(p)$ cones of apertures less than $60^\circ$ based at any given point.  Take this point to be $\bbx_1$. Then within each cone, there can be at most one $\bbx_j$, which is not equal to $\bbx_1$, and  whose nearest neighbor is $\bbx_1$. This shows that there can be at most $C(p)$ points distinct from $\bbx_1$ whose nearest neighbor is $\bbx_1$, completing the proof of the lemma.  
\end{proof}
\begin{lmm}\label{cplmm}
There is a constant $C(p)$ depending only on $p$, such that for any measurable $f:\rr^p \to [0,\infty)$ and any $n$, $\ee(f(\bbx_{n,1}))\le C(p)\ee(f(\bbx_1))$.
\end{lmm}
\begin{proof}
Since $f$ is nonnegative,
\begin{align*}
\ee(f(\bbx_{n,i})) &\le \ee(f(\bbx_i)) + \ee(f(\bbx_{n,i})1_{\{\bbx_{n,i}\ne X_i\}})\\
&\le \ee(f(\bbx_i)) + \sum_{j=1}^n \ee(f(\bbx_j)1_{\{\bbx_j = \bbx_{n,i},\, \bbx_j\ne \bbx_i\}}).
\end{align*}
Therefore by symmetry,
\begin{align*}
\ee(f(\bbx_{n,1})) &= \frac{1}{n}\sum_{i=1}^n \ee(f(\bbx_{n,i}))\\
&\le \frac{1}{n}\sum_{i=1}^n \ee(f(\bbx_i)) + \frac{1}{n}\sum_{i=1}^n\sum_{j=1}^n \ee(f(\bbx_j)1_{\{\bbx_j = \bbx_{n,i},\, \bbx_j\ne X_i\}})\\
&= \ee(f(\bbx_1)) + \frac{1}{n}\sum_{j=1}^n\ee\biggl(f(\bbx_j)\sum_{i=1}^n 1_{\{\bbx_j = \bbx_{n,i},\, \bbx_j\ne \bbx_i\}}\biggr)\\
&\le \ee(f(\bbx_1)) + \frac{1}{n}\sum_{j=1}^n\ee(f(\bbx_j)K_{n,j}) = \ee(f(\bbx_1)(1+K_{n,1})).
\end{align*}
By Lemma \ref{geomlmm}, this completes the proof. 
\end{proof}

For the next result, we will need the following version of Lusin's theorem (proved, for example, by combining \cite[Theorem 2.18 and Theorem 2.24]{rudin87}). 
\begin{lmm}[Special case of Lusin's theorem]\label{lusinthm}
Let $f:\rr^p\to \rr$ be a measurable function and $\gamma$ be a probability measure on $\rr^p$. Then, given any $\ve>0$, there is a compactly supported continuous function $g:\rr^p\to \rr$ such that $\gamma(\{\bx: f(\bx)\ne g(\bx)\}) <\ve$.
\end{lmm}
\begin{lmm}\label{nnthm}
For any measurable $f:\rr^p\to\rr$, $f(\bbx_1)-f(\bbx_{n,1})$ tends to $0$ in probability as $n\to\infty$. 
\end{lmm}
\begin{proof}
Fix some $\ve>0$. Let $g$ be a function as in Lemma~\ref{lusinthm}, for the given $f$ and $\ve$, and $\gamma=$ the law of $\bbx_1$. Then note that for any $\delta >0$, 
\begin{align*}
&\pp(|f(\bbx_1)-f(\bbx_{n,1})|>\delta) \\
&\le \pp(|g(\bbx_1)-g(\bbx_{n,1})|>\delta)+ \pp(f(\bbx_1)\ne g(\bbx_1)) \\
&\qquad + \pp(f(\bbx_{n,1})\ne g(\bbx_{n,1})).
\end{align*}
By Lemma \ref{nnlmm} and the continuity of $g$, 
\[
\lim_{n\to \infty} \pp(|g(\bbx_1)-g(\bbx_{n,1})|>\delta) =0.
\]
By the construction of $g$,
\[
\pp(f(\bbx_1)\ne g(\bbx_1))< \ve.
\]
Finally, by Lemma \ref{cplmm}, 
\[
\pp(f(\bbx_{n,1})\ne g(\bbx_{n,1})) \le C(p) \pp(f(\bbx_1)\ne g(\bbx_1))\le C(p)\ve.
\]
Putting it all together, we get
\[
\limsup_{n\to\infty} \pp(|f(\bbx_1)-f(\bbx_{n,1})|>\delta)\le \ve + C(p) \ve.
\]
Since $\ve$ and $\delta$ are arbitrary, this completes the proof of the lemma.
\end{proof}

Let $(Y_1,\bbx_1),\ldots,(Y_n, \bbx_n)$ be i.i.d.~copies of $(Y,\bbx)$. Let $F_n$ be the empirical distribution function of $Y_1,\ldots,Y_n$, that is,
\begin{equation*}
F_n(t) = \frac{1}{n}\sum_{i=1}^n1_{\{Y_i\le t\}}.
\end{equation*}
Also let
\begin{equation*}
G_n(t) = \frac{1}{n}\sum_{i=1}^n1_{\{Y_i\ge t\}}.
\end{equation*}
For each $i$, let $N(i)$ be the index $j$ such that $\bbx_j = \bbx_{n,i}$ (ties broken at random). 
Define 
\begin{equation}\label{qndef3}
Q_n = Q_n(Y,\bbx) := \frac{1}{n}\sum_{i=1}^n (\min\{F_n(Y_i), F_n(Y_{N(i)})\} - G_n(Y_i)^2). 
\end{equation}
Note that this is exactly the statistic $Q_n(Y,\bbx)$ defined in equation \eqref{qndef2} of Section \ref{resultsec2}. 
\begin{lmm}\label{expthm}
Let $Q_n$ be defined as above. Then
\[
\lim_{n\to\infty} \ee(Q_n(Y,\bbx)) = Q(Y, \bbx). 
\] 
\end{lmm}
\begin{proof}
Let  
\begin{equation}\label{qnpdef}
Q_n' := \frac{1}{n}\sum_{i=1}^n (\min\{F(Y_i), F(Y_{N(i)})\} - G(Y_i)^2).
\end{equation}
and let 
\[
\Delta_n := \sup_{t\in \rr} |F_n(t)-F(t)| + \sup_{t\in \rr} |G_n(t)-G(t)|.
\]
Then by the triangle inequality, 
\begin{equation}\label{qnqn}
|Q_n' - Q_n|\le 3\Delta_n.
\end{equation}
On the other hand, by the Glivenko--Cantelli theorem, $\Delta_n \to 0$ almost surely as $n\to\infty$. Since $\Delta_n$ is bounded by $2$, this implies that 
\[
\lim_{n\to \infty} \ee|Q_n'-Q_n|=0.
\]
Thus, it suffices to show that $\ee(Q_n')$ converges to $Q(Y,\bbx)$. First, notice that
\begin{align*}
\min\{F(Y_1),F(Y_{N(1)})\} &= \int 1_{\{ Y_1\ge t\}} 1_{\{Y_{N(1)}\ge t\}}d\mu(t).
\end{align*}
Let $\mf$ be the $\sigma$-algebra generated by $\bbx_1,\ldots,\bbx_n$ and the random variables used for breaking ties in the selection of nearest neighbors. Then for any $t$,
\begin{align*}
\ee(1_{\{Y_1\ge t\}} 1_{\{Y_{N(1)}\ge t\}}|\mf) &= G_{\bbx_1}(t) G_{\bbx_{N(1)}}(t). 
\end{align*}
Note that $\bbx_{N(1)} = \bbx_{n,1}$. Also, recall that by the properties of the regular conditional probability $\mu_\bx$, the map $\bx \mapsto G_\bx(t)$ is measurable. Therefore by the above identity and Lemma~\ref{nnthm}, we have 
\begin{align*}
\lim_{n\to\infty}\ee (1_{\{Y_1\ge t\}} 1_{\{Y_{N(1)}\ge t\}}) &= \ee(G_\bbx(t)^2).
\end{align*}
Thus,
\begin{align*}
\lim_{n\to \infty} \ee(Q_n')&=\int (\ee(G_\bbx(t)^2) - G(t)^2) d\mu(t).
\end{align*}
Since $\ee(G_\bbx(t))=G(t)$, this completes the proof of the lemma. 
\end{proof}
\begin{lmm}\label{concthm}
There are positive constants $C_1$ and $C_2$ depending only on the dimension $p$ such that for any $n$ and any $t\ge 0$,
\[
\pp(|Q_n - \ee(Q_n)|\ge t) \le C_1e^{-C_2nt^2}.
\]
\end{lmm}
\begin{proof}
Throughout this proof, $C(p)$ will denote any constant that depends only on $p$. The value of $C(p)$ may change from line to line. 

In addition to the variables $\bbx_i$ and $Y_i$, in this proof we will make use of i.i.d.~Uniform$[0,1]$ random variables $U_1,\ldots, U_n$, where $U_i$ is used for breaking ties if $\bbx_i$ has multiple nearest neighbors.

Our plan is to use  the bounded difference concentration inequality~\cite{mcdiarmid89}. For that, we have to get a bound on the maximum possible change in $Q_n$  if one $(Y_i, \bbx_i, U_i)$ is replaced by some alternative value $(Y_i', \bbx_i', U_i')$. We first write $Q_n = A_n + B_n$, where
\[
A_n := \frac{1}{n}\sum_{i=1}^n \min\{F_n(Y_i), F_n(Y_{N(i)})\}, \ \ \ B_n := \frac{1}{n}\sum_{i=1}^n G_n(Y_i)^2.
\] 
It is not hard to see that after the above replacement, each $G_n(Y_j)$ can change by at most $1/n$, and since these quantities are in $[0,1]$, $B_n$ can change by at most $2/n$. Therefore the bounded difference inequality gives
\begin{align}\label{bntail}
\pp(|B_n-\ee(B_n)|\ge t) \le 2e^{-nt^2/8}.
\end{align}
Unfortunately, $A_n$ is not well-behaved with respect to this kind of  perturbation, so we have to first replace $A_n$ by some more manageable quantity. Take a realization of $(Y_1,\bbx_1,U_1),\ldots,(Y_n, \bbx_n, U_n)$. Define an equivalence relation on $\{1,\ldots,n\}$ by declaring that $i$ and $j$ are equivalent if $\bbx_i=\bbx_j$. Call an equivalence class a `cluster' if its size is greater than one, and a `singleton' otherwise. Note that if $i$ belongs to a cluster $\mc$, then $N(i)$ must necessarily be also a member of the same cluster. In fact, $N(i)$ would be chosen uniformly at random (using $U_i$) from $\mc\setminus\{i\}$. 

Let $\fc$ denote the set of all clusters and $\fs$ denote the set of all singletons. For convenience, let us define 
\[
a_{i,j} := \min\{F_n(Y_i), F_n(Y_j)\},
\]
so that
\[
A_n = \frac{1}{n}\sum_{\mc\in \fc} \sum_{i\in \mc}a_{i,N(i)} + \frac{1}{n}\sum_{i\in \fs} a_{i,N(i)}. 
\]
Let $\mg$ denote the $\sigma$-algebra generated by $(Y_1,\bbx_1),\ldots, (Y_n,\bbx_n)$ and $(U_i)_{i\in \fs}$. Define $A_n' := \ee(A_n|\mg)$. Then it is clear that
\begin{align}\label{anpdef}
A_n' &= \frac{1}{n}\sum_{\mc\in \fc} b(\mc) + \frac{1}{n}\sum_{i\in \fs} a_{i,N(i)},
\end{align}
where
\[
b(\mc) := \frac{1}{|\mc|-1}\sum_{i\in \mc}\sum_{j\in \mc\setminus\{i\}}a_{i,j}. 
\]
We will now use the bounded difference inequality to get a tail bound for the difference $A_n-A_n'$. Conditional on $\mg$, $A_n$ is a function of $(U_i)_{i\notin\fs}$. If one such $U_i$ is replaced by some other value $U_i'$, then only $N(i)$ may be affected. Thus, $A_n$ changes by at most $1/n$. Therefore, the bounded difference inequality gives 
\[
\pp(|A_n-A_n'|\ge t|\mg) \le 2e^{-nt^2/2}. 
\]
Since the right side is deterministic, we can remove the conditioning on the left. But then the tail bound gives $\ee|A_n - A_n'|< 3n^{-1/2}$. Therefore,
\begin{align}
&\pp(|A_n-\ee(A_n)|\ge 3n^{-1/2} + t) \notag\\
&\le \pp(|A_n-A_n'|\ge t/2) + \pp(|A_n'-\ee(A_n')|\ge t/2) \notag\\
&\le 2e^{-nt^2/8}+ \pp(|A_n'-\ee(A_n')|\ge t/2).\label{antail}
\end{align}
So we now need to get a tail bound for $A_n'-\ee(A_n')$. Fortunately, $A_n'$ is well-behaved with respect to perturbing one coordinate. Let us now try to figure out the maximum possible change in $A_n'$ if some $(Y_i,\bbx_i, U_i)$ is replaced by an alternative value $(Y_i', \bbx_i', U_i')$. We will do this in stages. First, let us replace $\bbx_i$ by $\bbx_i'$, keeping $Y_i$ and $U_i$ fixed. We know by Lemma \ref{geomlmm} that in any configuration, for any $i$ there can be at most $C(p)$ singletons $j$ such that $i$ is a nearest neighbor of $j$ (not necessarily the chosen one). This fact will be used many times in the following argument.  There are several cases to consider:
\begin{enumerate}
\item Suppose that $i$ is in some cluster $\mc$ of size $\ge 3$ in the original configuration, and lands up in some other cluster $\mc'$ in the new configuration. Then the set of singletons is the same in the two configurations. If $j$ is a singleton, then $N(j)$ can change only if $i$ is a nearest neighbor of $j$ in either the original configuration or the final configuration. As noted above, there can be at most $C(p)$ such $j$. Therefore, due to these changes, $A_n'$ can change by at most $C(p)/n$. On the other hand, $b(\mc)$ changes by at most $2$, as seen from the following computation:
\begin{align*}
&|b(\mc) - b(\mc\setminus\{i\})| \\
&= \biggl|\frac{1}{|\mc|-1}\sum_{j\in \mc}\sum_{k\in \mc\setminus\{j\}} a_{j,k} - \frac{1}{|\mc|-2}\sum_{j\in \mc\setminus\{i\}}\sum_{k\in \mc\setminus\{i,j\}} a_{j,k}\biggr|\\
&= \biggl|\frac{1}{|\mc|-1}\sum_{k\in \mc\setminus\{i\}} a_{i,k} + \frac{1}{|\mc|-1} \sum_{j\in \mc\setminus\{i\}} a_{j,i}\\
&\qquad \qquad \qquad  - \frac{1}{(|\mc|-1)(|\mc|-2)}\sum_{j\in \mc\setminus\{i\}}\sum_{k\in \mc\setminus\{i,j\}} a_{j,k}\biggr|\le 2,
\end{align*}
where the last inequality holds because the $a_{i,j}$'s are in $[0,1]$. 
A similar calculation shows that $b(\mc')$ also changes by at most $1$. Thus, overall, $A_n'$ changes by at most $C(p)/n$.
\item Suppose that $i$ is in some cluster $\mc$ of size $\ge 3$ in the original configuration, and pairs up with a singleton to form a new cluster in the new configuration. Again, $b(\mc)$ changes by at most $2$, and the contributions from the singletons in \eqref{anpdef} changes by at most $C(p)/n$, by the same logic as in case (1). The formation of the new cluster causes a change of at most $2/n$. Therefore, again, the change in $A_n'$ is at most $C(p)/n$.
\item Suppose that $i$ is in some cluster $\mc$ of size $\ge 3$ in the original configuration, and becomes a singleton in the new configuration. Then just as before, $b(\mc)$ changes by at most $2$, and the contributions from singletons changes by at most $C(p)/n$. 
\item Suppose that $i$ is in some cluster $\mc$ of size $2$ in the original configuration, and pairs up with a singleton to form a new cluster in the new configuration. Again, the number of singletons $j$ for which $N(j)$ changes due to this operation is bounded by $C(p)$, and the contributions from the clusters terms in \eqref{anpdef} also changes by at most a bounded amount. Thus, the change in $A_n'$ is at most $C(p)/n$.
\item Suppose that $i$ is in some cluster $\mc$ of size $2$ in the original configuration, and becomes a singleton in the new configuration. Proceeding as before, we see that $A_n'$ changes by at most $C(p)/n$. 
\item Suppose that $i$ is a singleton in the original configuration and remains so in the new configuration. Again, it is clear that the change in $A_n'$ is at most $C(p)/n$.
\item All other cases are just reverses of the situations considered above. For example, if $i$ is a singleton in the original configuration and becomes part of a cluster of size $\ge 3$ in the new configuration, that's just the reverse of case (3). 
\end{enumerate}
Thus, we conclude that changing $\bbx_i$ to $\bbx_i'$ changes $A_n'$ by at most $C(p)/n$. Next, let us change $Y_i$ to $Y_i'$. Then $F_n(Y_j)$ changes by at most $1/n$ for each $j\ne i$, and $F_n(Y_i)$ changes by at most $1$. Therefore each $a_{j,k}$ changes by at most $1/n$ if $j\ne i$ and $k\ne i$, and by at most $1$ if either index equals $i$. From this it is easy to see that $A_n'$ can change by at most $1/n$. Finally, let us replace $U_i$ by $U_i'$. Then only $N(i)$ can change, and hence $A_n'$ can change by at most $1/n$. Combing all three steps, we get
\[
\pp(|A_n'-\ee(A_n')|\ge t) \le 2e^{-C(p)nt^2}. 
\]
Therefore by \eqref{bntail} and \eqref{antail}, we get
\begin{align*}
\pp(|A_n-\ee(A_n)|\ge 3n^{-1/2}+t) \le 6e^{-C(p)nt^2}. 
\end{align*}
If $t\ge 3n^{-1/2}$, this  bound  holds for $\pp(|A_n-\ee(A_n)|\ge 2t)$. If $t<3n^{-1/2}$, we can choose $C_1\ge 6$ so that $C_1e^{-C(p) nt^2} \ge 1$, so that it is trivially a bound for $\pp(|A_n-\ee(A_n)|\ge 2t)$. This completes the proof.
\end{proof}
Combining Lemmas \ref{expthm} and \ref{concthm}, we get the following corollary.
\begin{cor}\label{convascor}
As $n\to \infty$, $Q_n(Y,\bbx)\to Q(Y,\bbx)$ almost surely. 
\end{cor}

\section{Proof of Theorem \ref{basethm}}
Note that convergence of $Q_n(Y, \bbz)$ to the deterministic limit $c$ is the result of Corollary \ref{convascor} (applied to the pair $(Y,\bbz)$ instead of $(Y, \bbx)$). Showing that $S_n(Y)$ converges to $d$ is easier. Let
\[
S_n'(Y) = \frac{1}{n}\sum_{i=1}^n G(Y_i)(1 - G(Y_i)),
\]
and 
\[
\Delta_n := \sup_{t\in \rr} |G_n(t)-G(t)|.
\]
Then by triangle inequality $|S_n(Y) - S_n'(Y)| \leq 4\Delta_n$, and by the Glivenko--Cantelli theorem  $\Delta_n\rightarrow 0$ almost surely. So it is enough to show that $S_n'(Y)$ converges almost surely to $d$. But that is a consequence of the strong law of large numbers, since the $Y_i$'s are i.i.d and 
\[
\ee(G(Y_i)(1 - G(Y_i))) = \int G(t)(1 - G(t))d\mu(t) = d.
\] 
This completes the proof of the convergence claims in the theorem. Next, by combining Corollary~\ref{convascor} and Lemma~\ref{indepthm1}, we see that if $Y$ and $\bbx$ are independent, then $c=0$. This proves claim (i) in the theorem. On the other hand, if $Y$ is a function of $\bbz$, say $Y = f(\bbz)$ almost surely, then 
\begin{align*}
c &= \int\var (\pp(Y \geq t| \bbz))d\mu(t) \\
&= \int \var (\ee (1_{\{Y\geq t\}}|\bbz))d\mu(t)\\
&= \int \var (1_{\{f(\bbz)\geq t\}})d\mu(t)\\
&= \int \ee (1_{\{f(\bbz)\geq t\}})(1 - \ee(1_{\{f(\bbz)\geq t\}})) d\mu(t) = d,
\end{align*}
which proves claim (ii) in the theorem. 
Finally, by the law of total variance we have
\[
\var(1_{\{Y\geq t\}}) = \ee(\var(1_{\{Y\geq t\}}| \bbz)) + \var(\pp(Y\geq t| \bbz)),
\]
therefore $0\leq c\leq d$. Note that by Lemma \ref{indepthm1}, $c = 0$ if and only if $Y$ is independent of $\bbz$. To complete the proof of claim (iii), we have to show that if $c = d$ then $Y$ is almost surely a function of $\bbz$. If $c = d$, then 
\[
\int \ee(G_\bbz(t) - G_\bbz(t)^2) d\mu(t) = 0,
\]
which implies that $\pp(E)=1$, where $E$ is the event
\begin{equation}\label{edef}
\int G_\bbz(t)(1 - G_\bbz(t)) d\mu(t) = 0.
\end{equation}
Let $A$ be the support of $\mu$. Define
\[
a_{\bbz} := \sup\{t: G_\bbz(t) = 1\},  \ \ \ b_{\bbz} := \inf\{t: G_\bbz(t)=0\},
\]
so that $a_{\bbz}\le b_{\bbz}$. Now suppose that the event $\{a_{\bbz}<b_{\bbz}\}\cap E$ takes place. Since $G_\bbz(t)\in (0,1)$ for all $t\in (a_\bbz, b_\bbz)$, the condition \eqref{edef} implies that $\mu((a_\bbz, b_\bbz)) = 0$. Since $(a_\bbz, b_\bbz)$ is an open interval, this shows that $(a_\bbz, b_\bbz) \subseteq A^c$.  On the other hand, under the given circumstance, we also have $\pp(Y\in (a_\bbz, b_\bbz)|\bbz) > 0$. Thus $\pp(Y\in A^c|\bbz)>0$. 

The above argument implies that if $\pp(\{a_\bbz< b_\bbz\}\cap E)>0$, then $\pp(Y\in A^c)>0$. But this is impossible, since $A$ is the support of $\mu$. Therefore  $\pp(\{a_\bbz< b_\bbz\}\cap E)=0$. But $\pp(E)=1$. Therefore $\pp(a_{\bbz} = b_{\bbz})=1$. This implies that $Y$ is almost surely a function of $\bbz$. 

\section{Proof of Theorem \ref{mainthm}}
For the proof of Theorem \ref{mainthm}, we need some additional lemmas.

\begin{lmm}\label{expthmcond}
Let $Q_n(Y, \bbz|\bbx)$ be defined as in \eqref{qndef1}. Then $Q_n(Y, \bbz|\bbx)$ converges to $Q(Y, \bbz|\bbx)$ almost surely as $n\to\infty$, where
\begin{equation*}
Q(Y, \bbz| \bbx) := \int \ee(\var(G_{\bbw}(t)|\bbx)) d\mu(t),
\end{equation*}
where, as before, $\bbw = (\bbx, \bbz)$.
\end{lmm} 
\begin{proof}
Note that $Q_n(Y,\bbz|\bbx) = Q_n(Y,\bbw)-Q_n(Y,\bbx)$. Also,
\[
\ee(G_\bbw(t)|\bbx) = G_\bbx(t),
\]
which, by the law of total variance, gives
\begin{align*}
\var(G_\bbw(t)) - \var(G_\bbx(t)) &= \ee(\var(G_\bbw(t)|\bbx)). 
\end{align*}
Thus, 
\[
Q(Y,\bbz|\bbx) = Q(Y,\bbw)-Q(Y,\bbx).
\]
The result now follows by Corollary \ref{convascor}. 
\end{proof}
\begin{lmm}\label{expthmcondS}
For $S_n(Y, \bbx)$ defined in \eqref{sndef1}, 
\[
\lim_{n\rightarrow\infty}\ee(S_n(Y, \bbx)) = S(Y, \bbx)
\]
where $S(Y, \bbx) := \int\ee(\var(1_{\{Y\geq t\}}| \bbx))d\mu(t)$.
\end{lmm}
\begin{proof}
The proof uses techniques developed in the proof of Lemma~\ref{expthm}. Let 
\[
S_n'(Y, \bbx) = \frac{1}{n}\sum_{i=1}^n(F(Y_i) - \min\{F(Y_i), F(Y_{N(i)})\}),
\]
and 
\[
\Delta_n := \sup_{t\in\rr}|F_n(t) - F(t)|.
\]
By the triangle inequality, 
\[
|S_n'(Y, \bbx) - S_n(Y, \bbx)|\leq 4\Delta_n.
\]
By the Glivenko--Cantelli theorem, $\Delta_n\rightarrow 0$ almost surely and since $\Delta_n$ is bounded by 1, we can conclude that,
\[
\lim_{n\rightarrow\infty}\ee|S_n'(Y, \bbx) - S_n(Y, \bbx)| = 0.
\]
Then it is enough to show that $\ee (S_n'(Y, \bbx))$ converges to $S(Y, \bbx)$. Proceeding as in the proof of Lemma~\ref{expthm}, we get
\begin{align*}
\lim_{n\rightarrow\infty}\ee (S_n'(Y, \bbx)) &= \int( G(t) - \ee(G_{\bbx}(t)^2)) d\mu(t)\\
&= \int \ee(G_{\bbx}(t) - G_{\bbx}(t)^2) d\mu(t) = S(Y, \bbx),
\end{align*}
which completes the proof. 
\end{proof}

\begin{lmm}\label{concthmcondS}
There are positive constants $C_1$  and $C_2$ depending only on $p$ such that for any $n$ and any $t \geq 0$, 
\[
\pp(|S_n(Y, \bbx) - \ee(S_n(Y, \bbx))| \geq t) \leq C_1e^{-C_2nt^2}
\]
\end{lmm}
\begin{proof}
The concentration for the second term in the definition \eqref{sndef1} was already argued in the proof of Lemma~\ref{concthm}. For the first term, a simple application of the bounded difference inequality suffices. 
\end{proof}

Finally, we are ready to prove Theorem \ref{mainthm}.

\begin{proof}[Proof of Theorem \ref{mainthm}]
Convergence of $Q_n(Y, \bbz|\bbx)$ almost surely to $a = Q(Y, \bbz| \bbx)$ is the content of Lemma \ref{expthmcond}, and convergence of $S_n(Y,\bbx)$ to $b= S(Y,\bbx)$ follows by Lemmas \ref{expthmcondS} and \ref{concthmcondS}. 

Let us now prove the claims (i), (ii) and (iii) of the theorem.  First, let us prove (i).  It is not hard to see that $a = Q(Y,\bbw) - Q(Y,\bbx)$. Thus if $Y$ and $\bbz$ are conditionally independent given $\bbx$, then by Lemma~\ref{monotonethm}, $a=0$. 
This proves (i). Next, note that  
\begin{align*}
b  - a &= \int \ee(\var(1_{\{Y\geq t\}}| \bbx) - \var(\ee (1_{\{Y\geq t\}}|\bbz, \bbx)| \bbx))d\mu(t)\\
&= \int\ee(\ee( \var(1_{\{Y\geq t\}}| \bbz, \bbx)|\bbx)) d\mu(t)\\
&= \int \ee(\var(1_{\{Y\geq t\}}|\bbz,\bbx)) d\mu(t).
\end{align*}
Now, if with probability one $Y$ is a function of $\bbz$ conditional on $\bbx$, then $\var(1_{\{Y\geq t\}}| \bbz, \bbx) = 0$ almost surely. Thus, the above expression shows that $a = b$ in this situation.

Finally, let us prove claim (iii). Note that the above expression for $b-a$ also shows that $0 \leq a\leq b$, since $\var(1_{\{Y\geq t\}}| \bbz, \bbx)\geq 0$. Thus, it suffices to prove the opposite implications for (i) and (ii). 

If  $a = 0$, then again by Lemma \ref{monotonethm}, we get that $Y$ and $\bbz$ are conditionally independent given $\bbx$. If $a = b$, then there exists a set $A\subseteq\rr$ such that $\mu(A) = 1$ and for any $t\in A$ we have 
\[
\var(1_{\{Y\geq t\}}| \bbz, \bbx) = 0
\]
almost surely. Proceeding as the last part of the proof of Theorem \ref{basethm}, we can now conclude that $Y$ is almost surely equal to a function of $\bbw$. This implies that $Y$ is almost surely a function of $\bbz$ conditional on $\bbx$. 
\end{proof}

\section{Proof of Theorem \ref{ratethm}}
Throughout this section, we will assume that the assumptions (A1) and (A2) from Section \ref{ratesec} hold. In the following lemma, $\bbx_{n,1}$ is the nearest neighbor of $\bbx_1$ among $\bbx_2,\ldots,\bbx_n$ (with ties broken at random), as in previous sections. 
\begin{lmm}\label{nndistlmm0}
Let $C_1$ and $C_2$ be as in assumption \textup{(A2)}. Then there is some $C_3$ depending only on $C_1$, $C_2$ and $p$ such that
\begin{align*}
\ee(\min\{\|\bbx_1 - \bbx_{n,1}\|,1\})\le
\begin{cases}
C_3 n^{-1}(\log n)^3  & \text{ if } p=1,\\
C_3 n^{-1/p} (\log n)^{p+1} &\text{ if } p\ge 2.
\end{cases}
\end{align*}
\end{lmm}
\begin{proof}
Throughout this proof, $C$ will denote any constant that depends only on $C_1$, $C_2$ and $p$. 
Take any $t > 0$ and $\ve\in (n^{-1/p},1)$. Let $B$ be the ball of radius $t$ in $\rr^p$ centered at the origin. Partition $B$ into at most $Ct^p \ve^{-p}$ small sets of diameter $\le \ve$. Let $S$ be the small set containing $\bbx_1$. Then
\begin{align*}
\pp(\|\bbx_1-\bbx_{n,1}\|\ge \ve) &\le \pp(\bbx_1\notin B) + \pp(\bbx_2\notin S,\ldots,\bbx_n\notin S).
\end{align*}
Now note that
\begin{align*}
\pp(\bbx_2\notin S,\ldots,\bbx_n\notin S|\bbx_1)&= (1-\pp(\bbx_2\in S|\bbx_1))^{n-1} = (1-\nu(S))^{n-1},
\end{align*}
where $\nu$ is the law of $\bbx$. Let $A$ be the collection of all small sets with $\nu$-mass less than $\delta$. Since there are at most $Ct^p \ve^{-p}$ small sets, we get
\begin{align*}
\ee[(1-\nu(S))^{n-1}] &\le (1-\delta)^{n-1} + \pp(\bbx_1\in A)\\
&\le (1-\delta)^{n-1} + Ct^p \ve^{-p} \delta.
\end{align*}
Since $\pp(\bbx_1\notin B)\le C_1e^{-C_2t}$, this gives 
\begin{align*}
\pp(\|\bbx_1-\bbx_{n,1}\|\ge \ve) &\le C_1e^{-C_2t} + (1-\delta)^{n-1} + Ct^p \ve^{-p} \delta.
\end{align*}
Now choosing $\delta = Kn^{-1}\log n$ and $t = K \log( n\ve^p)$ for some large enough $K$, we get
\begin{align*}
\pp(\|\bbx_1-\bbx_{n,1}\|\ge \ve) &\le \frac{C (\log n)^{p+1}}{n\ve^p}. 
\end{align*}
Thus, 
\begin{align*}
\ee(\min\{\|\bbx_1-\bbx_{n,1}\|, 1\}) &= n^{-1/p}+\int_{n^{-1/p}}^1 \pp(\|\bbx_1-\bbx_{n,1}\|\ge \ve) d\ve\\
&\le  n^{-1/p} + \frac{C(\log n)^{p+1}}{n} \int_{n^{-1/p}}^1 \ve^{-p} d\ve.
\end{align*}
This is bounded by $Cn^{-1}(\log n)^3$ if $p=1$, and $Cn^{-1/p}(\log n)^{p+1}$ if $p\ge 2$.
\end{proof}

In the next lemma, let $Q=Q(Y,\bbx)$ be defined as in equation \eqref{qdef} and $Q_n = Q_n(Y,\bbx)$ be defined as in equation \eqref{qndef2}. 
\begin{lmm}\label{qqnlmm0}
Let $C$ and $\beta$ be as in assumption \textup{(A1)} and $C_1$ and $C_2$ be as in assumption \textup{(A2)}. Then there are $K_1$, $K_2$ and $K_3$ depending only on $C$, $\beta$, $C_1$, $C_2$ and $p$ such that for any $t\ge 0$,
\[
\pp(|Q_n-Q|\ge K_1 n^{-\min\{1/p, 1/2\}}(\log n)^{p+\beta+1} + t ) \le K_2e^{-K_3nt^2}.
\]
\end{lmm}
\begin{proof}
Let $Q_n'$ and $\Delta_n$ be as in the proof of Lemma~\ref{expthm}. By the Dvoretzky--Kiefer--Wolfowitz inequality~\cite{dkw56, massart90}, we know that for any $x\ge0$,
\begin{align*}
\pp(\sqrt{n}\Delta_n \ge x) \le 2e^{-2x^2}. 
\end{align*}
From this it follows that $\ee(\Delta_n)\le n^{-1/2}$, and therefore by \eqref{qnqn},
\begin{align}\label{qnqnprime0}
\ee|Q_n'-Q_n|\le 3n^{-1/2}. 
\end{align} 
Arguing as in the proof of Lemma \ref{expthm}, we get 
\begin{align*}
\ee(Q_n') &= \int (\ee(G_{\bbx_1}(t) G_{\bbx_{n,1}}(t)) -G(t)^2) d\mu(t). 
\end{align*}
On the other hand, 
\begin{align*}
Q = \int (\ee(G_{\bbx}(t)^2)-G(t)^2) d\mu(t). 
\end{align*}
Since $G_\bx(t)\in [0,1]$ for all $\bx$ and $t$, this gives 
\begin{align*}
|\ee(Q_n')-Q| &\le \int \ee|G_{\bbx_1}(t) - G_{\bbx_{n,1}}(t)| d\mu(t).
\end{align*}
Now note that by assumption (A1),
\begin{align*}
|G_{\bbx_{n,1}} (t) - G_{\bbx_{1}} (t)|  &\le C(1+\|\bbx_{n,1}\|^\beta + \|\bbx_1\|^\beta) \|\bbx_1-\bbx_{n,1}\|.
\end{align*}
Next, note that by assumption (A2), $\pp(\|\bbx_1\|\ge t)\le C_1e^{-C_2t}$. Therefore by Lemma~\ref{cplmm}, the law of $\|\bbx_{n,1}\|$ also has an exponentially decaying tail. Lastly, note that $|G_{\bbx_{n,1}} (t) - G_{\bbx_{1}} (t)|\le 1$. So, letting $E$ be the event that the maximum of $\|\bbx_1\|$ and $\|\bbx_{n,1}\|$ is bigger than $K\log n$ for some suitably large $K$, we get
\begin{align*}
\ee|G_{\bbx_{n,1}} (t) - G_{\bbx_{1}} (t)| &\le \pp(E) + \ee(|G_{\bbx_{n,1}} (t) - G_{\bbx_{1}} (t)|1_{E^c})\\
&\le n^{-1} + L (\log n)^\beta \ee(\min\{\|\bbx_1-\bbx_{n,1}\|,1\})
\end{align*}
for some large constant $L$. It is now easy to complete the proof using Lemma~\ref{nndistlmm0}, inequality \eqref{qnqnprime0}, and  Lemma \ref{concthm}.
\end{proof} 

We are now ready to prove Theorem \ref{ratethm}.
\begin{proof}[Proof of Theorem \ref{ratethm}]
Recall from Section \ref{mainpfsec} that 
\[
T_n(Y,\bbz|\bbx) = \frac{Q_n(Y, \bbz|\bbx)}{S_n(Y,\bbx)},
\]
and 
\[
T(Y,\bbz|\bbx) = \frac{Q(Y, \bbz|\bbx)}{S(Y,\bbx)},
\]
where the quantity $Q(Y,\bbz|\bbx)$ is defined in Lemma \ref{expthmcond} and $S(Y,\bbx)$ is defined in Lemma~\ref{expthmcondS}. 
Now, as we observed in the proof of Lemma \ref{expthmcond}, 
\[
Q_n(Y, \bbz|\bbx) = Q_n(Y,\bbw) - Q_n(Y,\bbx),
\]
where $\bbw = (\bbx, \bbz)$. Therefore by Lemma \ref{qqnlmm0}, 
\[
Q_n(Y, \bbz|\bbx)  - Q(Y, \bbz|\bbx) = O_P\biggl(\frac{(\log n)^{p+q+\beta+1}}{n^{1/(p+q)}}\biggr). 
\]
By an exactly similar argument, 
\[
S_n(Y, \bbx)  - S(Y, \bbx) = O_P\biggl(\frac{(\log n)^{p+\beta+1}}{n^{\min\{1/p,1/2\}}}\biggr). 
\]
Finally, by part (iv) of Theorem \ref{mainthm}, $S(Y,\bbx)\ne 0$. The proof is completed by combining these observations. 
\end{proof}

\section{Proof of Proposition \ref{conditionprop}}
Take a bounded open ball $B$  in $\rr^p$ and let $K$ be the closure of $B$. Let $g(y) := \max_{\bx\in K}f(y|\bx)$. By assumption, $g(y)$ is bounded and decays faster than any negative power of $|y|$ as $|y|\to\infty$. Also, since $K$ is bounded, the assumption on the derivatives of $\log f(y|\bx)$ implies that 
\[
h(y) := \max_{\bx\in K} \biggl|\fpar{}{x_i}\log f(y|\bx)\biggr|
\]
is bounded above by a polynomial in $|y|$. Thus, for $\bx\in K$ and $y\in \rr$,
\begin{align*}
\biggl|\fpar{}{x_i} f(y|\bx)\biggr| &= \biggl|f(y|\bx) \fpar{}{x_i}\log f(y|\bx)\biggr|\\
&\le g(y)h(y),
\end{align*}
and $g(y)h(y)$ is an integrable function of $y$. This allows us to apply the dominated convergence theorem and conclude that for any $\bx\in B$ (and hence any $\bx\in \rr^p$), 
\begin{align*}
\biggl|\fpar{}{x_i} \pp(Y\ge t|\bbx=\bx)\biggr| &= \biggl|\fpar{}{x_i}\int_t^\infty f(y|\bx) dy\biggr|\\
&= \biggl|\int_t^\infty \fpar{}{x_i}f(y|\bx) dy\biggr|\\
&\le \int_{-\infty}^\infty \biggl|\fpar{}{x_i}\log f(y|\bx)\biggr| f(y|\bx)dy.
\end{align*}
Now applying the assumption about the derivatives of $\log f(y|\bx)$, and the condition that $\ee(Y^{2k}|\bbx=\bx)$ is bounded by a polynomial in $\|\bx\|$ for any $k$, it follows easily that 
\[
\biggl|\fpar{}{x_i} \pp(Y\ge t|\bbx=\bx)\biggr|
\]
is bounded above by a polynomial in $\|\bx\|$. The second inequality in (A1) follows directly from this.

\section{Proof of Theorem \ref{selectthm}}
Let $j_1,j_2,\ldots,j_p$ be the complete ordering of all variables produced by the stepwise algorithm in FOCI. Let $S_0:=\emptyset$, and for each $1\le k\le p$, let $S_k := \{j_1,\ldots, j_k\}$. For  $k> p$, let $S_k := S_p$. For any subset $S$, let $Q(Y,\bbx_S)$ be defined as in \eqref{qdef} and let $Q_n(Y,\bbx_S)$ be defined as in~\eqref{qndef2}. Notice that $Q(Y,\bbx_S)$ is the same as the quantity $Q(S)$ defined in~\eqref{qsdef}. Define these quantities to be zero if $S=\emptyset$. Let $K$ be the integer part of $1/\delta+2$. Let $E'$ be the event that $|Q_n(Y, \bbx_{S_k})-Q(Y,\bbx_{S_k})|\le \delta/8$ for all $1\le k\le K$, and let $E$ be the event that $S_K$ is sufficient. 
\begin{lmm}\label{newlmm1}
Suppose that $E'$ has happened, and also that 
\begin{equation}\label{qncond}
Q_n(Y, \bbx_{S_k})-Q_n(Y, \bbx_{S_{k-1}}) \le \frac{\delta}{2}
\end{equation}
for some $1\le k\le K$. Then $S_{k-1}$ is sufficient.
\end{lmm}
\begin{proof}
Take any $k\le K$ such that \eqref{qncond} holds. If $k>p$ there is nothing to prove. So let us assume that $k\le p$. An examination of the formula for $T_n$ shows that for each $k$, $j_k$ is the index $j$ that maximizes $Q_n(Y, \bbx_{S_{k-1}\cup\{j\}})$ among all $j\notin S_{k-1}$.  Since $E'$ has happened, this implies that for any $j\notin S_{k-1}$, 
\begin{align*}
Q(Y, \bbx_{S_{k-1}\cup\{j\}}) - Q(Y, \bbx_{S_{k-1}}) &\le Q_n(Y, \bbx_{S_{k-1}\cup\{j\}})-Q_n(Y, \bbx_{S_{k-1}}) + \frac{\delta}{4}\\
&\le Q_n(Y, \bbx_{S_k})-Q_n(Y, \bbx_{S_{k-1}}) + \frac{\delta}{4}\\
&\le \frac{3\delta}{4}.
\end{align*}
Therefore since $\delta>0$, the definition of $\delta$ implies that $S_{k-1}$ must be a sufficient subset of predictors.
\end{proof}
\begin{lmm}\label{elmm}
 The event $E'$ implies $E$.
\end{lmm}
\begin{proof}
Suppose that $E'$ has happened. Suppose also that \eqref{qncond} is violated for every $1\le k\le K$. Since $E'$ has happened, this implies that for each  $k\le K$, 
\begin{align*}
Q(Y, \bbx_{S_k}) - Q(Y, \bbx_{S_{k-1}}) &\ge Q_n(Y, \bbx_{S_k})-Q_n(Y, \bbx_{S_{k-1}}) - \frac{\delta}{4}\\
&\ge \frac{\delta}{4}.
\end{align*}
This gives 
\begin{align*}
Q(Y, \bbx_{S_K}) &= \sum_{k=1}^K (Q(Y, \bbx_{S_k}) - Q(Y, \bbx_{S_{k-1}})) \\
&\ge \frac{K\delta}{4}\ge \biggl(\frac{1}{\delta}+1\biggr)\frac{\delta}{4} > \frac{1}{4}.
\end{align*}
But the variance of any  $[0,1]$-valued random variable is bounded by $1/4$, which implies that $1/4$ is the maximum possible value of the statistic $Q$. This yields a contradiction, proving that \eqref{qncond} must hold for some $k\le K$. Therefore by Lemma \ref{newlmm1}, $S_K$ is sufficient.
\end{proof}

\begin{lmm}\label{elmm2}
There are positive constants $L_1$, $L_2$ and $L_3$ depending only on  $C$, $\beta$, $C_1$, $C_2$ and $K$, such that 
\[
\pp(E')\ge 1- L_1p^{L_2}e^{-L_3n}. 
\]
\end{lmm}
\begin{proof}
Throughout this proof, $L_1,L_2,\ldots$ will denote constants that depend only on $C$, $\beta$, $C_1$, $C_2$ and $K$. By assumptions (A1$'$) and (A2$'$), and Lemma~\ref{qqnlmm0}, there exist $L_1$, $L_2$ and $L_3$ such that for any $S$ of size $\le K$ and any $t\ge 0$, 
\begin{align*}
&\pp(|Q_n(Y, \bbx_S) - Q(Y,\bbx_S)|\ge L_1 n^{-\min\{1/K, 1/2\}}(\log n)^{K+\beta+1} + t) \\
&\le L_2e^{-L_3nt^2}. 
\end{align*}
Call the event on the left $A_{S,t}$. Let
\[
A_t := \bigcup_{|S|\le K} A_{S,t}.
\]
Then by a simple union bound,
\[
\pp(A_t)\le L_2p^K e^{-L_3nt^2}. 
\]
Now choose $t = \delta/16$. If $n$ is so large that 
\begin{equation}\label{nlower}
L_1 n^{-\min\{1/K, 1/2\}}(\log n)^{K+\beta+1}\le \frac{\delta}{16},
\end{equation} 
then the above bound implies that 
\begin{equation}\label{pebd}
\pp(E') \ge 1- L_2p^K e^{-L_4n}.
\end{equation}
Now, the condition \eqref{nlower} can be written as $n\ge L_5$. Choose a constant $L_6\ge L_2$ so large  that for any $n<L_5$,
\[
L_6p^K e^{-L_3n} \ge 1.
\]
Then if $n<L_5$, we have $\pp(E') \ge 1- L_6p^K e^{-L_3n}$. Combining with \eqref{pebd}, we see that this inequality holds without any constraint on $n$.
\end{proof}
\begin{lmm}\label{newlmm2}
The event $E'$ implies that $\hat{S}$ is sufficient.
\end{lmm}
\begin{proof}
Suppose that $E'$ has happened. Consider two cases. First, suppose that FOCI has stopped at step $K$ or later. Then $S_K\subseteq \hat{S}$. By Lemma \ref{elmm}, $E$ has also happened, and hence $S_K$ is sufficient. Therefore in this case, $\hat{S}$ is sufficient. Next, suppose that FOCI has stopped at step $k-1< K$. Then by the definition of the stopping rule, we see that 
\[
Q_n(Y, \bbx_{S_k}) \le Q_n(Y, \bbx_{S_{k-1}}). 
\]
In particular, \eqref{qncond} holds. Since $E'$ has happened, Lemma \ref{newlmm1} now implies that $\hat{S} = S_{k-1}$ is sufficient.
 \end{proof}
It is clear that  Lemmas \ref{elmm2} and \ref{newlmm2} together imply Theorem \ref{selectthm}.


\section{Proof of Theorem \ref{deltacompare}}\label{gaussianproof}
We start with the following lemma about the variance of a certain kind of function of normal random variables. 
\begin{lmm}\label{normbdlmm}
Let $\Phi$ be the standard normal c.d.f.~and let $Z\sim N(0,1)$. There are positive constants $C_1$ and $C_2$ such that for any $a,b\in \rr$,
\[
C_1 b^2 e^{-(a^2+b^2)} \le \var(\Phi(a+bZ)) \le C_2 b^2. 
\]
\end{lmm}
\begin{proof}
Since $Z$ has the same law as $-Z$, and $\Phi(-x)=1-\Phi(x)$ for all $x$, it is easy to see that there is no loss of generality in assuming that $a$ and $b$ are nonnegative. Moreover, since the result is trivial when $b=0$, let us also assume that $b>0$. Let $Z,Z'$ be i.i.d.~$N(0,1)$ random variables, so that
\[
\var(\Phi(a+bZ)) = \frac{1}{2}\ee[(\Phi(a+bZ)-\Phi(a+bZ'))^2]. 
\]
Let $\varphi = \Phi'$ be the standard normal p.d.f. Now, 
\[
\Phi(a+bZ)-\Phi(a+bZ') = \varphi(Y) b(Z-Z')
\]
for some $Y$ lying between $a+bZ$ and $a+bZ'$. Suppose that in a particular realization, $Z$ and $Z'$ both turn out to be in $[-1,1]$. Then $Y$ lies between $a-b$ and $a+b$, which implies that $\varphi(Y)$ is at least as large as the minimum of $\varphi(a-b)$ and $\varphi(a+b)$. Thus, in this situation,
\begin{align*}
\varphi(Y) &\ge \frac{1}{\sqrt{2\pi}} \exp\biggl(-\frac{1}{2}\max\{(a-b)^2, (a+b)^2\}\biggr)\\
&\ge \frac{1}{\sqrt{2\pi}} \exp(-(a^2+b^2)).
\end{align*}
This shows that 
\begin{align*}
\var(\Phi(a+bZ)) &\ge \frac{1}{2}\ee[(\Phi(a+bZ)-\Phi(a+bZ'))^2; |Z|\le 1, |Z'|\le 1]\\
&\ge C_1 b^2e^{-(a^2+b^2)} \ee[(Z-Z')^2; |Z|\le 1, |Z'|\le 1]\\
&= C_2 b^2 e^{-(a^2+b^2)},
\end{align*}
where $C_1$ and $C_2$ are positive universal constants. 
This proves the lower bound. For the upper bound, simply observe that since $\varphi$ is uniformly bounded by $1/\sqrt{2\pi}$, we have
\begin{align*}
\var(\Phi(a+bZ)) &\le \frac{b^2}{4\pi}\ee[(Z - Z')^2]= \frac{b^2}{2\pi}.
\end{align*}
This completes the proof of the lemma.
\end{proof}
The next lemma compares one of our measures of conditional dependence with partial $R^2$ in the case of normal random variables.
\begin{lmm}\label{gaussianlmm}
Let $(Y, \bbx, \bbz)$ be jointly normal, with $Y\sim N(0,\tau^2)$ for some $\tau >0$. Let $Q(Y,\bbz|\bbx)$ be defined as in the statement of Lemma \ref{expthmcond}.  Let $\alpha^2 := \var(Y|\bbx, \bbz)$ and $\beta^2 := \var(Y|\bbx)$, and assume that these numbers are nonzero. Let $R^2_{Y,\bbz|\bbx}$ be the partial $R^2$ of $Y$ and $\bbz$ given $\bbx$. There are positive universal constants $C_1$ and $C_2$ such that
\begin{align*}
\frac{C_1\beta^2 e^{-\beta^2/\alpha^2}}{\alpha\tau}R_{Y,\bbz|\bbx}^2 \le Q(Y,\bbz|\bbx) &\le \frac{C_2\beta^2}{\alpha^2} R_{Y,\bbz|\bbx}^2.
\end{align*}
The same bounds hold if we replace $Q(Y, \bbz|\bbx)$ by $Q(Y,\bbz)$ (defined in equation \eqref{qdef}) and $R^2_{Y,\bbz|\bbx}$ by $R^2_{Y,\bbz}$ (the usual $R^2$ between $Y$ and $Z$) on both sides, and define $\beta^2$ as $\var(Y)$. 
\end{lmm}
\begin{proof}
Let $\bbw := (\bbx, \bbz)$, and for each $t\in \rr$, let
\begin{align*}
Y_t := \pp(Y\ge t|\bbw). 
\end{align*}
Let $U := \ee(Y|\bbw)$ and $V := \ee(Y|\bbx)$.  Given $\bbw$, $Y$ is normal with mean $U$ and variance $\alpha^2$. Thus,
\begin{align}
Y_t &= \pp((Y-U)/\alpha \ge (t-U)/\alpha|\bbw) \notag \\
&= 1-\Phi((t-U)/\alpha) = \Phi((U-t)/\alpha). \label{ytexp}
\end{align}
Recall that 
\begin{align*}
Q(Y,\bbz|\bbx) &= \int\ee(\var(Y_t|\bbx)) d\mu(t),
\end{align*}
where $\mu$ is the $N(0,\tau^2)$ probability measure. So, by \eqref{ytexp},
\begin{align}\label{qphi}
Q(Y,\bbz|\bbx) &=  \int\ee(\var(\Phi((U-t)/\alpha)|\bbx)) d\mu(t).
\end{align}
Now, note that $\ee(U|\bbx) = V$. Next, note that
\begin{align*}
\ee[(Y-U)(U-V)|\bbx] &= \ee[\ee((Y-U)(U-V)|\bbw)|\bbx]  = 0,
\end{align*}
since $\ee(Y|\bbw) = U$, and $U$ and $V$ are functions of $\bbw$. Thus,
\begin{align*}
\ee[(Y-V)^2|\bbx] &= \ee[(Y-U)^2|\bbx] + \ee[(U-V)^2|\bbx]\\
&=  \ee[\ee((Y-U)^2|\bbw)|\bbx] + \var(U|\bbx)\\
&= \ee[\var(Y|\bbw)|\bbx] + \var(U|\bbx)\\
&= \alpha^2 + \var(U|\bbx). 
\end{align*}
But $\ee[(Y-V)^2|\bbx] = \beta^2$. Thus, $\var(U|\bbx)=\beta^2 - \alpha^2$. Therefore, given $\bbx$, $U$ is normal with mean $V$ and variance $\beta^2 - \alpha^2$. Let
\[
b := \frac{\sqrt{\beta^2-\alpha^2}}{\alpha}. 
\]
Then, by \eqref{qphi}, we get
\begin{align*}
Q(Y,\bbz|\bbx) &=  \int\ee(\var(\Phi(b Z +  (V-t)/\alpha)|\bbx)) d\mu(t),
\end{align*}
where $Z$ is a standard normal random variable, independent of all else. By Lemma \ref{normbdlmm},
\begin{align*}
C_1 b^2 e^{-(V-t)^2/\alpha^2 - b^2}\le \var(\Phi(b Z +  (V-t)/\alpha)|\bbx) \le C_2 b^2.
\end{align*}
Plugging these bounds into the previous display, we get
\begin{align}\label{c1b}
C_1 b^2e^{-b^2} \int\ee( e^{-(V-t)^2/\alpha^2})d\mu(t)\le Q(Y,\bbz|\bbx)  \le C_2 b^2.
\end{align}
Now note that 
\begin{align*}
\int \ee(e^{-(V-t)^2/\alpha^2})d\mu(t) &= \ee(e^{-(V-\xi)^2/\alpha^2}),
\end{align*}
where $\xi\sim N(0,\tau^2)$ and is independent of $V$. But $V-\xi\sim N(0, 2\tau^2-\beta^2)$, since 
\begin{align*}
\var(V) &= \var(Y) - \ee(\var(Y|\bbx)) = \tau^2-\beta^2. 
\end{align*}
Thus, a simple computation gives
\begin{align*}
\ee(e^{-(V-\xi)^2/\alpha^2}) &= \biggl(1 + \frac{2(2\tau^2-\beta^2)}{\alpha^2}\biggr)^{-1/2}\\
&= \frac{\alpha}{\sqrt{\alpha^2 + 2(2\tau^2-\beta^2)}}\ge \frac{\alpha}{2\tau},
\end{align*}
where the last inequality holds because $\alpha \le \beta$. 
Plugging this lower bound into \eqref{c1b} and using $b^2\le \beta^2/\alpha^2$, we get
\begin{align*}
\frac{C_1\alpha e^{-\beta^2/\alpha^2}b^2}{2\tau} \le Q(Y,\bbz|\bbx) \le C_2 b^2. 
\end{align*}
Finally, observe that
\begin{align*}
b^2  &=  \frac{\beta^2}{\alpha^2} \frac{\beta^2-\alpha^2}{\beta^2}=\frac{\beta^2}{\alpha^2}R_{Y,\bbz|\bbx}^2
\end{align*}
and substitute this in the previous display. This completes the proof of the first assertion of the lemma. The second assertion follows similarly, by retracing the steps in the proof and making suitable changes at the appropriate places.
\end{proof}
We are now ready to prove Theorem \ref{deltacompare}.
\begin{proof}[Proof of Theorem \ref{deltacompare}]
Let $S$ be an insufficient subset of predictors. Then there is some $j\notin S$ such that $Q(Y, X_j|\bbx_S) \ge \delta$, because if $Q(S)$ is defined as in \eqref{qsdef}, then it is not hard to see that
\[
Q(Y, X_j|\bbx_S) = Q(S\cup \{j\}) - Q(S).
\]
But then, by Lemma \ref{gaussianlmm},
\begin{align*}
\rho(S, j) &= R^2_{Y,X_j|\bbx_S}\\
&\ge \frac{\alpha^2}{C_2\beta^2} Q(Y,X_j|\bbx_S) \ge \frac{\alpha^2\delta}{C_2\beta^2},
\end{align*}
where $\alpha^2 = \var(Y|\bbx_{S\cup\{ j\}})$ and $\beta^2 = \var(Y|\bbx_S)$. Since $\alpha^2 \ge \var(Y|\bbx) =\sigma^2$ and $\beta^2\le  \var(Y) = \tau^2$, this shows that 
\[
\rho(S, j)  \ge \frac{\sigma^2\delta}{C_2\tau^2}. 
\]
This proves that $\delta' \ge \sigma^2 \delta /C_2\tau^2$. Conversely, for any insufficient set $S$, there is some $j\notin S$ such that $\rho(S,j)\ge \delta'$. So by  Lemma \ref{gaussianlmm}, 
\begin{align*}
Q(Y, X_j|\bbx_S) &\ge \frac{C_1\beta^2 e^{-\beta^2/\alpha^2}\delta'}{\alpha\tau}.
\end{align*}
Now, $\beta^2 \ge \var(Y|\bbx) = \sigma^2$, $\alpha^2 \le \var(Y)=\tau^2$, and $\alpha^2\le \beta^2$. Thus,
\[
Q(Y, X_j|\bbx_S) \ge \frac{C_1\sigma^2 e^{-1}\delta'}{\tau^2}.
\]
Thus, $\delta\ge C_1e^{-1} \sigma^2 \delta'/\tau^2$. This completes the proof of the theorem.
\end{proof}

\section*{Acknowledgments}
We are grateful to Mohsen Bayati, Persi Diaconis, Adityanand Guntuboyina, Susan Holmes, Bodhisattva Sen and Rob Tibshirani for helpful comments, and to Nima Hamidi, Norm Matloff, and Balasubramanian Narasimhan for help with preparing the R package FOCI. We also thank the anonymous referees and the associate editor for various useful suggestions that helped improve the paper.

\end{document}